\documentclass[11pt]{amsart}
\usepackage{graphicx,amscd,amsmath,amsfonts,amssymb,color,geometry}

\usepackage[centertags]{amsmath}
\usepackage{amsthm}
\usepackage[colorlinks]{hyperref}

\usepackage[shortlabels]{enumitem} 

\usepackage[initials]{amsrefs}

\makeatletter
\@namedef{subjclassname@2020}{%
 \textup{2020} Mathematics Subject Classification}
\makeatother

\newtheorem{theorem}{Theorem}[section]
\newtheorem{proposition}[theorem]{Proposition}

\newtheorem{remark}[theorem]{Remark}

\newtheorem{lemma}[theorem]{Lemma}

\numberwithin{equation}{section}

\input epsf

\DeclareMathOperator{\sign}{sign}
\DeclareMathOperator{\graph}{graph}
\DeclareMathOperator{\ext}{ext}

\newcommand{\vertiii}[1]{{\left\vert\kern-0.25ex\left\vert\kern-0.25ex\left\vert #1 \right\vert\kern-0.25ex\right\vert\kern-0.25ex\right\vert}}
\begin{document}

\baselineskip=.65cm

\begin{abstract}
This work is a thorough and detailed study on the geometry of the unit sphere of certain Banach spaces of homogeneous polynomials in ${\mathbb{R}}^2$. Specifically, we provide a complete description of the unit spheres, identify the extreme points of the unit balls, derive explicit formulas for the corresponding polynomial norms, and describe the techniques required to tackle these questions.

To enhance the comprehensiveness of this work, we complement the results and their proofs with suitable diagrams and figures. The new results presented here settle some open questions posed in the past. For the sake of completeness of this work, we briefly discuss previous known results and provide directions of research and applications of our results.
\end{abstract}

\title{Geometry of homogeneous polynomials in ${\mathbb R}^2$}

\author[Garc\'{\i}a]{Domingo Garc\'{\i}a}
\address[Domingo Garc\'{\i}a]{\newline \indent Departamento de An\'{a}lisis Matem\'{a}tico,\newline \indent 
Universidad de Valencia,\newline \indent Doctor Moliner 50, 46100 Burjasot (Valencia),\newline \indent Spain}
\email{domingo.garcia@uv.es}

\author[Jung]{Mingu Jung}
\address[Mingu Jung]{\newline \indent June E Huh Center for Mathematical Challenges,\newline \indent Korea Institute for Advanced Study (KIAS),\newline \indent 02455 Seoul, Republic of Korea}
\email{jmingoo@kias.re.kr}

\author[Maestre]{Manuel Maestre}
\address[Manuel Maestre]{\newline \indent Departamento de An\'{a}lisis Matem\'{a}tico,\newline \indent 
	Universidad de Valencia, \newline \indent Doctor Moliner 50, 46100 Burjasot (Valencia), Spain}
\email{manuel.maestre@uv.es}

\author[Mu\~{n}oz]{Gustavo A. Mu\~{n}oz-Fern\'{a}ndez}
\address[Gustavo A. Mu\~{n}oz-Fern\'{a}ndez]{\newline \indent Instituto de Matem\'{a}tica Interdisciplinar (IMI),\newline \indent Departamento de An\'{a}lisis Matem\'{a}tico y Matem\'atica Aplicada,\newline \indent 
	Facultad de Ciencias Matem\'{a}ticas,\newline \indent 
	Plaza de Ciencias 3,\newline \indent 
	Universidad Complutense de Madrid,\newline \indent 
	Madrid, 28040 (Spain).}
\email{gustavo\_fernandez@mat.ucm.es}

\author[Seoane]{ Juan B.~Seoane-Sep\'{u}lveda}
\address[Juan B.~Seoane-Sep\'{u}lveda]{\newline \indent Instituto de Matem\'{a}tica Interdisciplinar (IMI),\newline \indent Departamento de An\'{a}lisis Matem\'{a}tico y Matem\'atica Aplicada,\newline \indent 
	Facultad de Ciencias Matem\'{a}ticas,\newline \indent 
	Plaza de Ciencias 3,\newline \indent 
	Universidad Complutense de Madrid,\newline \indent 
	Madrid, 28040 (Spain).}
\email{jseoane@ucm.es}

\thanks{The first and the third authors were supported by projects PID2021-122126NB-C33/MCIN/AEI/ 10.13039/ 501100011033\- (FEDER) and PROMETEU/2021/070. The second author was supported by KIAS Individual Grants (MG086601, HP086601) at Korea Institute for Advanced Study and June E Huh Center for Mathematical Challenges.\\
\textit{Corresponding author}: Mingu Jung. Email: jmingoo@kias.re.kr}


\maketitle

\tableofcontents

\section{Introduction}

The study of the unit ball of a polynomial space has motivated a significant volume of publications. However, this problem has much earlier roots. As early as 1966, Konheim and Rivlin \cite{KoRi} provided a characteristic property of the extreme points of the unit ball of the space ${\mathcal P}_n({\mathbb R})$, which consists of polynomials on the real line of degree at most $n$, endowed with the norm
$$
\|P\|=\sup\{|P(x)|: x\in[-1,1]\}.
$$
The search for characterizations of the extreme points of the unit ball of a polynomial space intensified since the late 1990's. The references \cite{GaGrMa,Di3,AK,BMRS,CK1,CK2,CK3,GMSS,ChKiKi,G1,G2,G3,G4,G5,G6,G7,MN1,MN2,MN3,MPSW,MRS,MS,N,
	RyTu,MR4740398} are just a selection of contributions to the study of the geometry of a number of polynomial spaces of low dimension. The interested reader can find a monograph on this topic in \cite{FGMMRS}.

Interestingly, polynomial norms are frequently non-absolute. Recall that a norm $\| \cdot \|$ on a linear subspace $X$ of ${\mathbb R}^\Lambda$ (here, $\Lambda$ is any non-empty set) is \textit{absolute} whenever it is complete and satisfies the following conditions:
\begin{enumerate}
	\item[(a)] Given $x, y \in \mathbb{R}^\Lambda$ with $|x(\lambda)|=|y(\lambda)|$ for every $\lambda \in \Lambda$, if $x \in X$, then $y \in X$ with $\|y\| = \|x\|$. 
	\item[(b)] For every $\lambda \in \Lambda$, the function $e_\lambda : \Lambda \rightarrow \mathbb{R}$ given by $e_\lambda (\xi) = \delta_{\lambda \xi}$ for $\xi \in \Lambda$, belongs to $X$ with $\|e_\lambda\|=1$. 
\end{enumerate}
One property which can be deduced from the definition is the following: given $x, y \in \mathbb{R}^\Lambda$ with $|y(\lambda)|\leq |x(\lambda)|$ for every $\lambda \in \Lambda$, if $x \in X$, then $y \in X$ with $\|y\| \leq \|x\|$. 
The $\ell_p$-norms together with many other classical norms are clearly absolute in ${\mathbb R}^\Lambda$. Hence polynomial norms can be viewed as a source of natural examples of non-absolute norms. 

Perhaps one of the most relevant motivations to study polynomial spaces, and more specifically the extreme points of their unit balls, rests on the fact that an elementary application of the Krein-Milman theorem allows us to obtain optimal constants in a number of polynomial inequalities of interest. For instance, in \cite{AMRS, AJMS, MSS, JMPS, MSanSeo1, MSanSeo2,V,AAMarkov,ENG,KoRi} one can find several applications that obtain optimal Bernstein and Markov type estimates using a combination of geometric results and the so-called Krein-Milman approach. Similarly, several polynomial Bohnenblust-Hille type inequalities (see, e.g., \cite{DGMPbook,MR3250297,MR4740398}) can be derived (see, for instance, \cite{JMMS, DMPS, MPS}). The same techniques can be applied to obtain sharp estimates of polarization constants and unconditional constants (see \cite{MRS, G7}) and many other polynomial inequalities of interest (see, for instance, \cite{AEMRS, AJMS2}).

We now provide a precise definition of the polynomial space that will be examined in this paper. Consider the space consisting of the homogeneous trinomials $ax^m+bx^{m-n}y^n+cy^m$ on the plane ${\mathbb R}^2$ where $a,b,c\in{\mathbb R}$ and $m,n\in{\mathbb N}$ are such that $m>n$. That space, endowed with the norm
	\begin{align*}
	\vertiii{ax^m+bx^{m-n}y^n+cy^m}_{m,n} &=\sup\left\{|ax^m+bx^{m-n}y^n+cy^m|:(x,y)\in[-1,1]^2\right\}
	\end{align*} 
will be denoted by ${\mathcal P}^h_{m,n}({\mathbb R})$. The aim of this paper is to explore the geometry of the unit ball of ${\mathcal P}^h_{m,n}({\mathbb R})$, extending and generalizing the results established by Choi, Kim, and Ki in 1998 \cite{ChKiKi}, Aron and Klimek in 2001 \cite{AK}, and Muñoz and Seoane in 2008 \cite{MS}. Furthermore, we complete the investigation of ${\mathcal P}^h_{m,n}({\mathbb R})$ that was initiated in \cite{JMR2021} and \cite{GMS2023}.

In particular, we will provide a complete description of the unit sphere ${\mathsf S}^h_{m,n}$ in ${\mathcal P}^h_{m,n}({\mathbb R})$ and characterize the extreme points of the unit ball ${\mathsf B}^h_{m,n}$ in ${\mathcal P}^h_{m,n}({\mathbb R})$. The solutions to these two questions depend significantly on whether $m$ and $n$ are even or odd. A thorough analysis reveals that the problem must be examined in the following three main cases:
\begin{enumerate}[ref=\Alph*]
	\item Case \ref{caseA}: $m$ is odd. \label{caseA}
	
	\item Case \ref{caseB}: Both $m$ and $n$ are even. \label{caseB}
	
	\item Case \ref{caseC}: $m$ is even and $n$ is odd. \label{caseC}
\end{enumerate}
Cases \ref{caseA} and \ref{caseB} have been previously studied in \cite{JMR2021} and \cite{GMS2023}, respectively. For the sake of completeness, we will present the known results for each case in the following.

\subsection{Case \ref{caseA}: $m$ is odd}\mbox{ }\newline
\indent For Case \ref{caseA}, the detailed study of the geometry of the unit sphere ${\mathsf S}^h_{m,n}$ of ${\mathcal P}^h_{m,n}({\mathbb R})$ needs to be divided into two different subcases, depending on whether $\frac{m}{n} < 2$ or $\frac{m}{n} > 2$ (since $\frac{m}{n} = 2$ is not possible in this case). Although the two subcases within Case \ref{caseA} are similar, let us briefly describe below what the geometry of the unit balls ``looks like."

As mentioned earlier, the study of the geometry of ${\mathsf B}_{m,n}$ depends strongly on whether $m$ and $n$ are even or odd and each of the four possible choices of the parity of $m$ and $n$ requires a specific treatment (see \cite{MS}). As a related space to ${ \mathcal{P}}^h_{m,n} (\mathbb{R})$, consider the space ${\mathcal P}_{m,n}({\mathbb R})$ consisting of the trinomials in the real line of the form $ax^m+bx^n+c$ endowed with the norm
	$$
	\|ax^m+bx^n+c\|_{m,n}=\sup\{|ax^m+bx^n+c|:x\in[-1,1]\}.
	$$
For simplicity, we will frequently use the representation of the polynomials $ax^m+bx^{m-n}y^n+cy^m \in {\mathcal P}^h_{m,n}({\mathbb R})$ or $ax^m+bx^n+c \in {\mathcal{P}}_{m,n} (\mathbb{R})$ as the triple $(a,b,c)\in{\mathbb R}^3$, that is, 
\begin{enumerate}
\item[] $\vertiii{(a,b,c)}_{m,n} = \sup\left\{|ax^m+bx^{m-n}y^n+cy^m|:(x,y)\in[-1,1]^2\right\}$, and 
\item[] $\| (a,b,c)\|_{m,n} = \sup\{|ax^m+bx^n+c|:x\in[-1,1]\}$.
\end{enumerate} 

The norms $\vertiii{\cdot}_{m,n}$ and $\| \cdot \|_{m,n}$ are tightly related; in fact, using the fact that the elements of ${\mathcal P}^h_{m,n}({\mathbb R})$ attain their norm on the set $\{(1,y),(x,1):x,y\in[-1,1]\}$ by homogeneity and symmetry, it is not difficult to see that 
\begin{equation}\label{eq:relation}
\vertiii{(a,b,c)}_{m,n}=\max\{\|(a,b,c)\|_{m,m-n},\|(c,b,a)\|_{m,n}\},
\end{equation}
for every $(a,b,c)\in{\mathbb R}^3$.
The previous identity reveals that 	
\begin{equation}\label{eq:reduction}
	\vertiii{(a,b,c)}_{m,n}=\vertiii{(c,b,a)}_{m,m-n},
\end{equation}
for every $(a,b,c)\in{\mathbb R}^3$. 
The identity \eqref{eq:reduction} allows us to simplify the study of the geometry of ${\mathsf B}^h_{m,n}$, at least when $m$ is odd. The case where both $m$ and $n$ are odd can be reduced to the case where $m$ is odd and $n$ is even by swapping $a$ and $c$ on the one hand, and $n$ and $m-n$ on the other. Thus, in order to describe the Case \ref{caseA}, we shall focus our attention on the case $m$ odd and $n$ even. First, let us present the following auxiliary result, followed by the explicit formulas for the norm $\vertiii{\cdot}_{m,n}$.

\begin{proposition}[\mbox{\cite[Lemma 2.1]{MS}}] \label{formulaOddOddThmLin}
If $m,n\in{\mathbb N}$ are such that $m>n$ then the equation
	$$
	|n+mx|=(m-n)|x|^\frac{m}{m-n}
	$$
	has only three roots, one at $x=-1$, another one at a point
	$\lambda_0\in(-\frac{n}{m},0)$ and a third one at a point
	$\lambda_1>0$. In addition to that we have
	\begin{equation}\label{convexity}
		|n+mx|< (m-n)|x|^\frac{m}{m-n},
	\end{equation}
if and only if $x<\lambda_0$ or $x > \lambda_1$.
\end{proposition}

\begin{remark}\label{mu0}

The dependence of $\lambda_0$, in Proposition \ref{formulaOddOddThmLin}, on $m$ and $n$ justifies the notation $\lambda_0(m,n)$ to represent $\lambda_0$. The value of $\lambda_0(m,m-n)$ for every odd number $m$ and every even number $n$ with $m>n$ will play an important role in the results for Case \ref{caseA}. For short, we put $\mu_0(m,n)=\lambda_0(m,m-n)$, or simply $\mu_0=\mu_0(m,n)$. 

\end{remark}

\begin{proposition}[\mbox{\cite[Theorem 3.4]{JMR2021}}]
\label{formulaOddThmHom}
	Let $m,n\in{\mathbb N}$ with $m>n$, $m$ odd and $n$ even. Consider the number $K_{m,n}=\frac{n}{m-n}\left(\frac{m-n}{m}\right)^\frac{m}{n}$, the interval $I_{m,n}=[\eta_1,\eta_2]$, where $\eta_1=-\frac{m}{m-n}$, $\eta_2=\frac{m}{m-n}\mu_0$ and $\mu_0=\mu_0(m,n)$ is the number in $(-\frac{m-n}{m},0)$ introduced in Proposition \ref{formulaOddOddThmLin}, and the sets $A_{m,n}$, $F_{m,n}$, $B_{m,n}$ and ${\mathcal B}$ (see Figures \ref{fig:mnmenor2} and \ref{fig:mnmayor2}) given by
\begin{align*}
	A_{m,n}&=\left\{(x,y)\in {\mathbb R}^2:x\in I_{m,n}\text{ and } |y|\ge 1-K_{m,n}|x|^\frac{m}{n}\right\},\\
	F_{m,n}&=\{(x,y)\in{\mathbb R}^2:x\in I_{m,n}\text{ and } 1-K_{m,n}|x|^\frac{m}{n}<|y|< 1-|1+x|\},\\
	{\mathcal B}&=\{(x,y)\in{\mathbb R}^2:|x+1|+|y|<1\},\\
	B_{m,n}&={\mathcal B}\setminus F_{m,n}.
\end{align*}
Then,
\begin{equation}\label{formulaOddOddHom}
	\vertiii{(a,b,c)}_{m,n} =
	\begin{cases}
		\frac{n|a|}{m-n} \left|\frac{(m-n)b}{ma}\right|^\frac{m}{n} + |c|&\text{if $a\ne 0$ and $\left(\frac{b}{a},\frac{c}{a}\right)\in A_{m,n}$},\\
		|a| & \text{if $a\ne 0$ and $\left(\frac{b}{a},\frac{c}{a}\right)\in B_{m,n}$},\\
		|a+b|+|c| & \text{otherwise}.
	\end{cases}
\end{equation}	
\end{proposition}

\begin{figure}
	\centering
	\includegraphics[height=.75\textwidth,keepaspectratio=true]{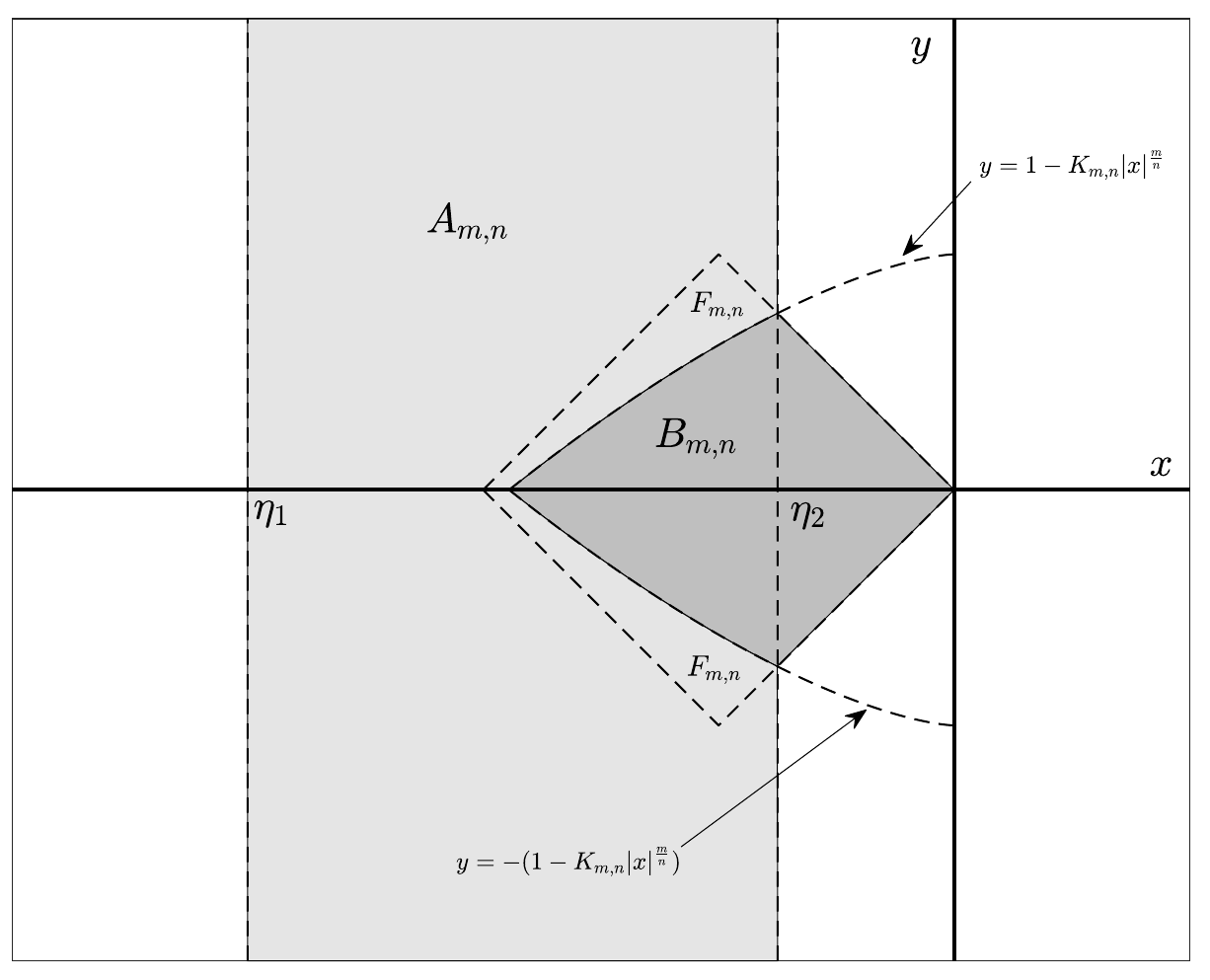}
	\caption{Regions appearing in the definition of $\vertiii{\cdot}_{m,n}$ in Case \ref{caseA}, where $m$ is odd, $n$ is even and $\frac{m}{n}<2$. The case considered in the picture corresponds to the choice $m=3$ and $n=2$.}\label{fig:mnmenor2}
\end{figure}

\begin{figure}
	\centering
	\includegraphics[height=.75\textwidth,keepaspectratio=true]{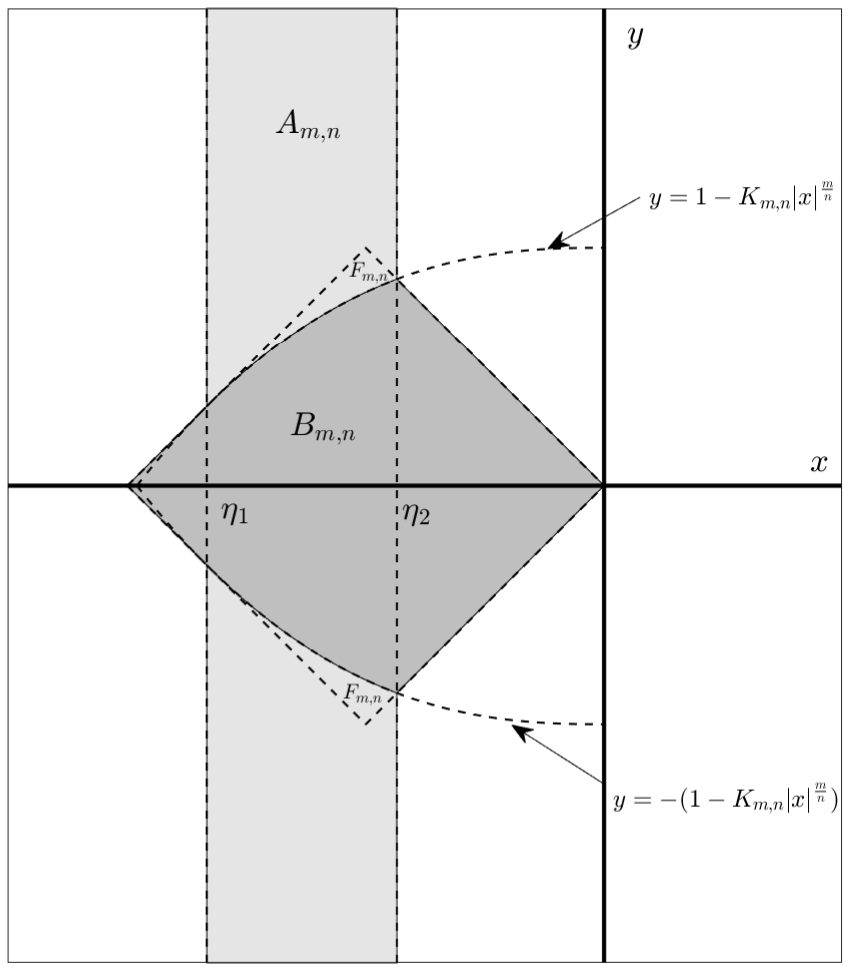}
	\caption{Regions appearing in the definition of $\vertiii{\cdot}_{m,n}$ in Case \ref{caseA}, where $m$ is odd, $n$ is even and $\frac{m}{n}>2$. The case considered in the picture corresponds to the choice $m=5$ and $n=2$.}\label{fig:mnmayor2}
\end{figure}

Once the formulas for the norm are provided, a different matter is to obtain the set of extreme points for the unit ball. In this Case \ref{caseA}, the description of extreme points of ${\mathsf B}^h_{m,n}$ are given in \cite{JMR2021}, namely:

\begin{proposition}[\mbox{\cite[Theorem 3.6]{JMR2021}}] \label{extodd}
	Let $m,n\in{\mathbb N}$ be such that $m>n$, $m$ is odd and $n$ is even and suppose that 
	$$K_{m,n}=\frac{n}{m-n}\left(\frac{m-n}{m}\right)^\frac{m}{n}, $$
	$$L_{m,n}=\frac{m}{m-n}\left(\frac{m-n}{n}\right)^\frac{n}{m},$$ 
	$$a_0=\frac{m-n}{n}, \quad \eta_1=-\frac{m}{m-n}, \quad \text{ and } \quad \eta_2=\frac{m}{m-n}\mu_0.$$
	We have: 
	\begin{itemize}
		\item[$\bullet$] If $\frac{m}{n}<2$, then 
		\begin{align*}
			\ext({\mathsf B}^h_{m,n})&=\left\{\pm\left(-1,t,\pm(1-K_{m,n}|t|^\frac{m}{n})\right):t\in[-\eta_2,L_{m,n}]\right\}\\
			&\quad \cup\left\{\pm (0,s,L_{m,n}|s|^\frac{m}{n}):s\in[-1,-a_0]\right\}\\
			&\quad \cup\{(\pm 1,0,0),(0,0\pm 1)\}.
		\end{align*}
			\item[$\bullet$] If $\frac{m}{n}>2$, then 
		\begin{align*}
			\ext({\mathsf B}^h_{m,n})&=\left\{\pm\left(-1,t,\pm(1-K_{m,n}|t|^\frac{m}{n})\right):t\in[-\eta_2,-\eta_1]\right\}\\
			&\quad \cup\{(\pm1,0,0),(0,0\pm 1),\pm(1,-2,0)\}.
		\end{align*}
	\end{itemize}
\end{proposition}

These extreme points are spotted in Figures \ref{figureball_54} and \ref{figureball_52} (which are represented for some choices of $m$ and $n$).

\begin{figure}[!h]
	\centering
	\includegraphics[height=.65\textwidth,keepaspectratio=true]{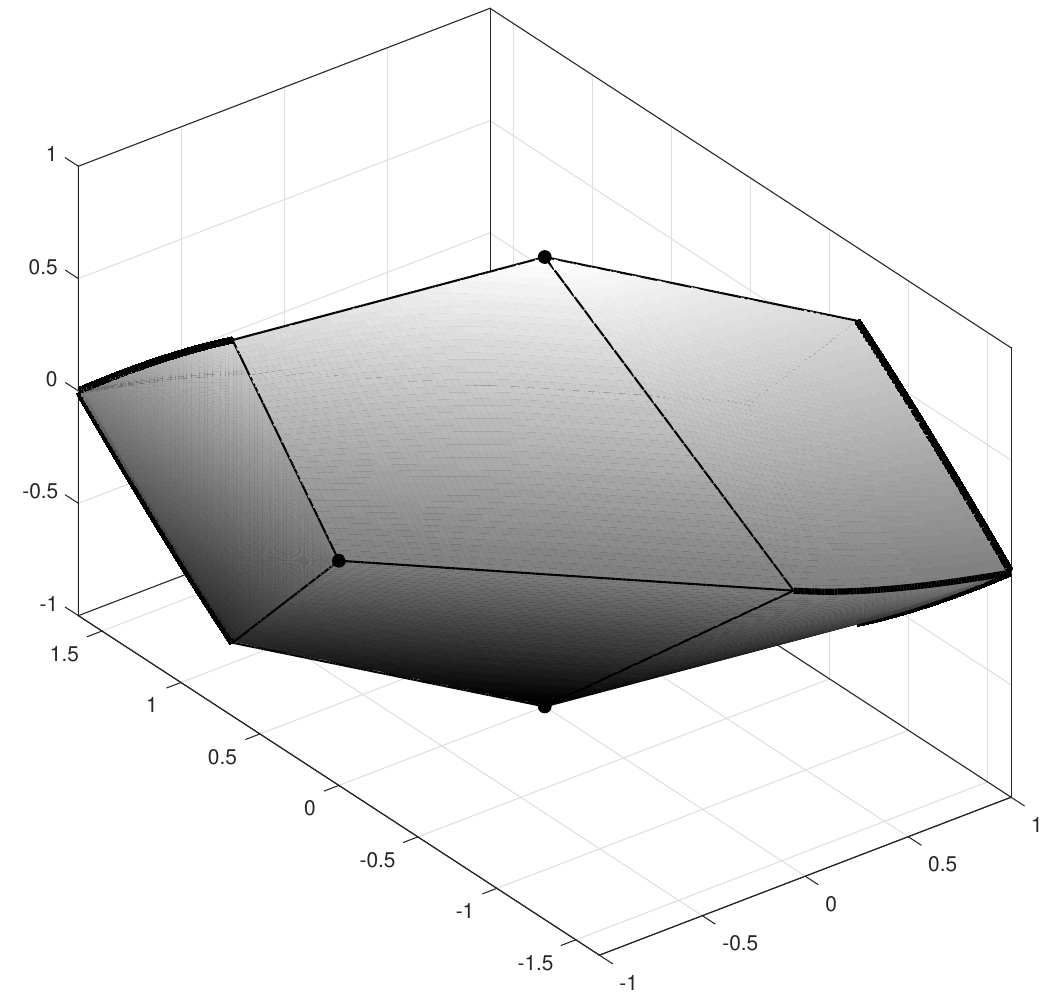}
	\caption{Sketch of ${\mathsf S}^h_{m,n}$ with $\frac{m}{n}<2$ in Case \ref{caseA}. The picture corresponds with the choice $m=5$, $n=4$. The extreme points appear with a thicker line or big dots. The surfaces that form ${\mathsf S}_{m,n}^h$ are delimited by thin lines.}\label{figureball_54}
\end{figure}

\begin{figure}[!h]
	\centering
	\includegraphics[height=.65\textwidth,keepaspectratio=true]{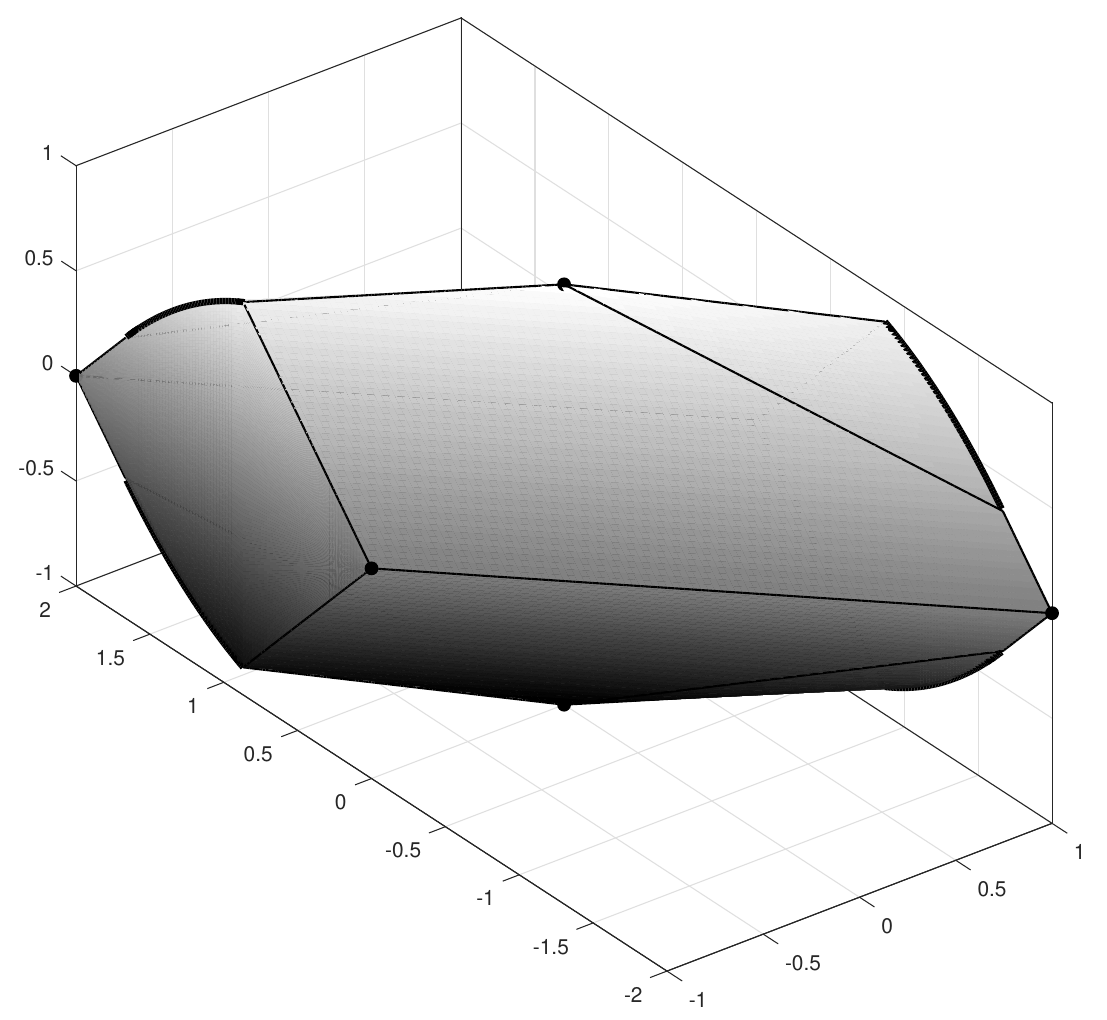}
	\caption{Sketch of ${\mathsf S}^h_{m,n}$ with $\frac{m}{n}>2$ in Case \ref{caseA}. The picture corresponds with the choice $m=5$, $n=2$. The extreme points appear with a thicker line or big dots. The surfaces that form ${\mathsf S^h}_{m,n}$ are delimited by thin lines.}\label{figureball_52}
\end{figure}


\subsection{Case \ref{caseB}: $m$ and $n$ are both even}\mbox{ }\newline
\indent Regarding Case \ref{caseB}, where both $m$ and $n$ are even, the study of the norm and geometry of the unit sphere is highly technical and complex. In this case, we need to consider three different situations: $\frac{n}{m} \in (0, 1/3)$, $\frac{n}{m} \in [1/3, 2/3]$, and $\frac{n}{m} \in (2/3, 1)$. The outcomes differ quite a bit depending on each situation. To make this paper self-contained, we will present these results as well, which can be found in \cite{GMS2023}.

First, regarding the explicit formula of $\vertiii{\cdot}_{m,n}$ in this case, we have the following:

\begin{proposition}[\mbox{\cite[Theroem 2.6]{GMS2023}}]\label{thm:formulaeveneven}
	Let $m,n\in{\mathbb N}$ be even with $m>n$. Consider the sets $A_{m,n}$, $B_{m,n}$, $C_{m,n}$, and $D_{m,n}$, (see Figures \ref{fig:nmMenor13} and \ref{fig:nmEntre13y23}). Provided that $K_{m,n}=\frac{n}{m-n}\left(\frac{m-n}{m}\right)^\frac{m}{n}$ and letting
	\begin{align*}
		\Gamma_{m,n}(a,b,c)&=K_{m,n}|a|\left|\frac{a}{b}\right|^{-\frac{m}{n}} + \sign(b) c,\\
		\Upsilon(a,b,c)&=\left|\frac{a+b}{2}+c\right|+\left|\frac{a+b}{2}\right|,
	\end{align*}
	for all $(a,b,c)\in{\mathbb R}^3$, then the explicit value for $\vertiii{(a,b,c)}_{m,n}$ is given by
		\begin{align*}
			\Gamma_{m,n}(a,b,c)&\text{ if $b\ne 0$ and $\left(\frac{a}{b},\frac{c}{b}\right)\in A_{m,n}$},\\
			\Gamma_{m,m-n}(c,b,a)&\text{ if $b\ne 0$ and $\left(\frac{a}{b},\frac{c}{b}\right)\in B_{m,n}$},\\
			\Upsilon(a,b,c) & \text{ if $b\ne 0$ and $\left(\frac{a}{b},\frac{c}{b}\right)\in C_{m,n}$ or $b=0$ and $|c|\geq |a|$},\\
			\Upsilon(c,b,a) & \text{ if $b\ne 0$ and $\left(\frac{a}{b},\frac{c}{b}\right)\in D_{m,n}$ or $b=0$ and $|c|\leq |a|$.}
		\end{align*}
\end{proposition}

\begin{figure}
	\centering
	\includegraphics[height=0.7\textwidth,keepaspectratio=true]{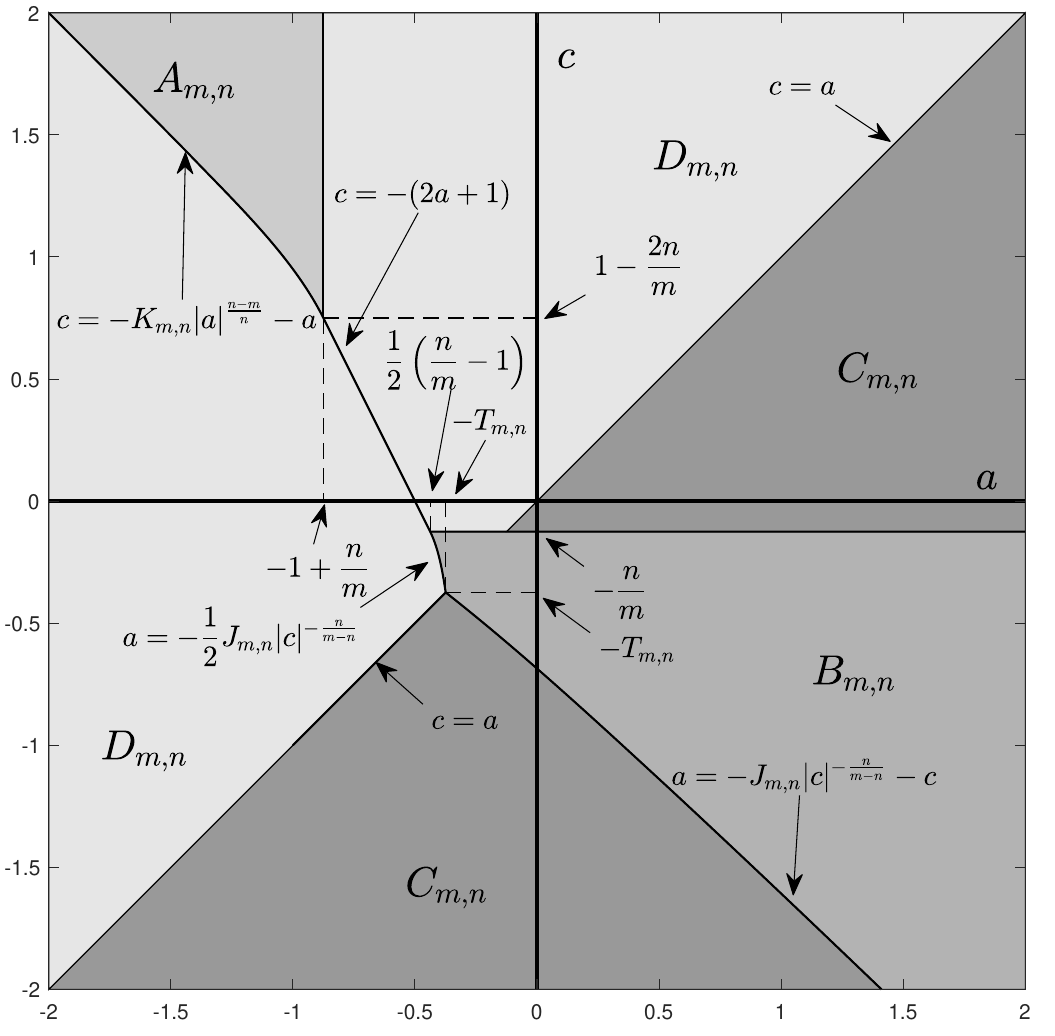}
	\caption{Representation of the sets $A_{m,n}$, $B_{m,n}$, $C_{m,n}$ and $D_{m,n}$ in Case \ref{caseB} for $\frac{n}{m}<\frac{1}{3}$. Here we have considered the case $m=16$ and $n=2$.}\label{fig:nmMenor13}
\end{figure}

\begin{figure}
	\centering
	\includegraphics[height=0.7\textwidth,keepaspectratio=true]{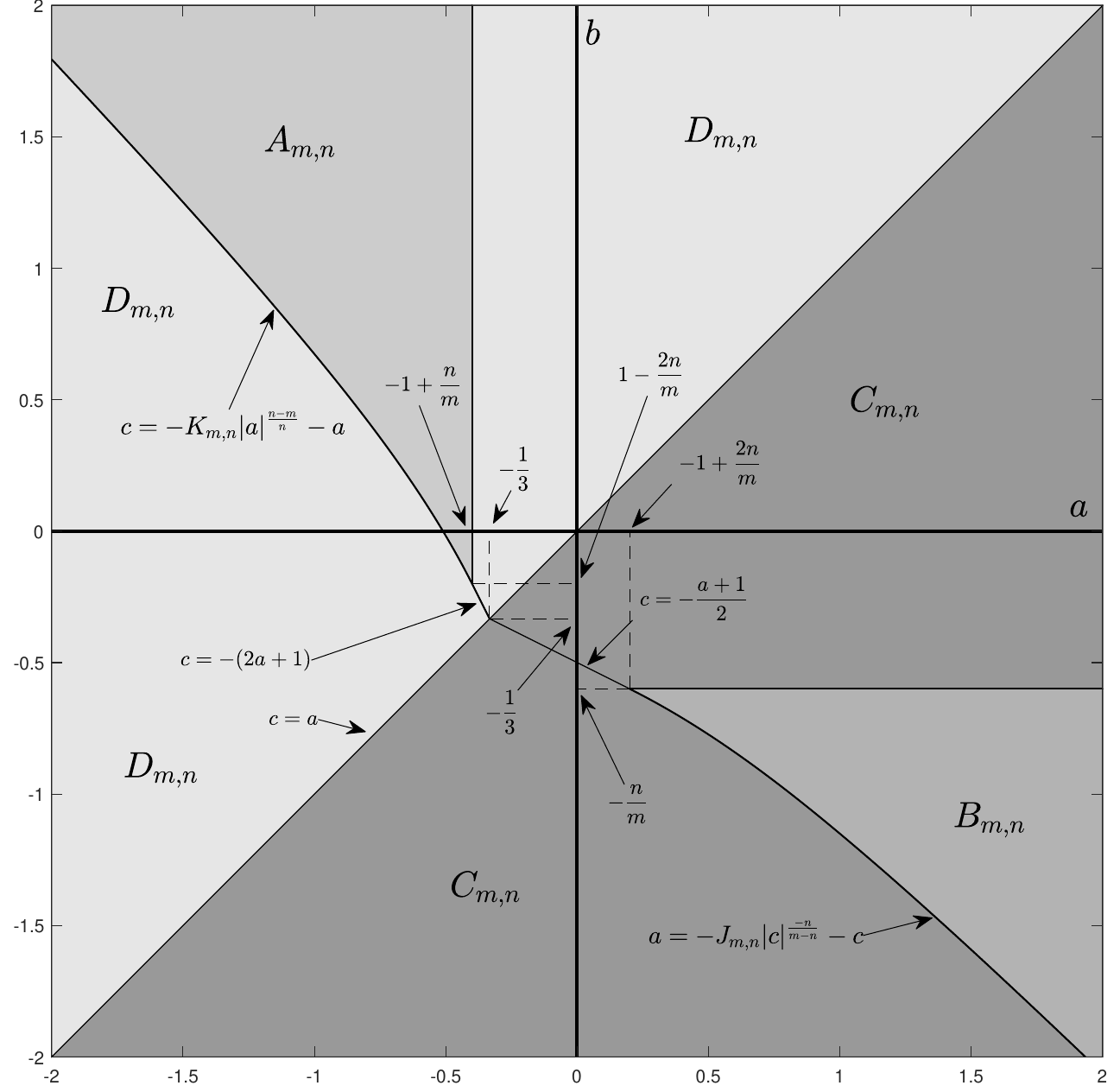}
	\caption{Representation of the sets $A_{m,n}$, $B_{m,n}$, $C_{m,n}$ and $D_{m,n}$ in Case \ref{caseB} for $\frac{1}{3}\leq\frac{n}{m}\leq\frac{2}{3}$. Here we have considered the case $m=30$ and $n=18$.}\label{fig:nmEntre13y23}
\end{figure}

\begin{figure}
	\centering
	\includegraphics[height=.7\textwidth,keepaspectratio=true]{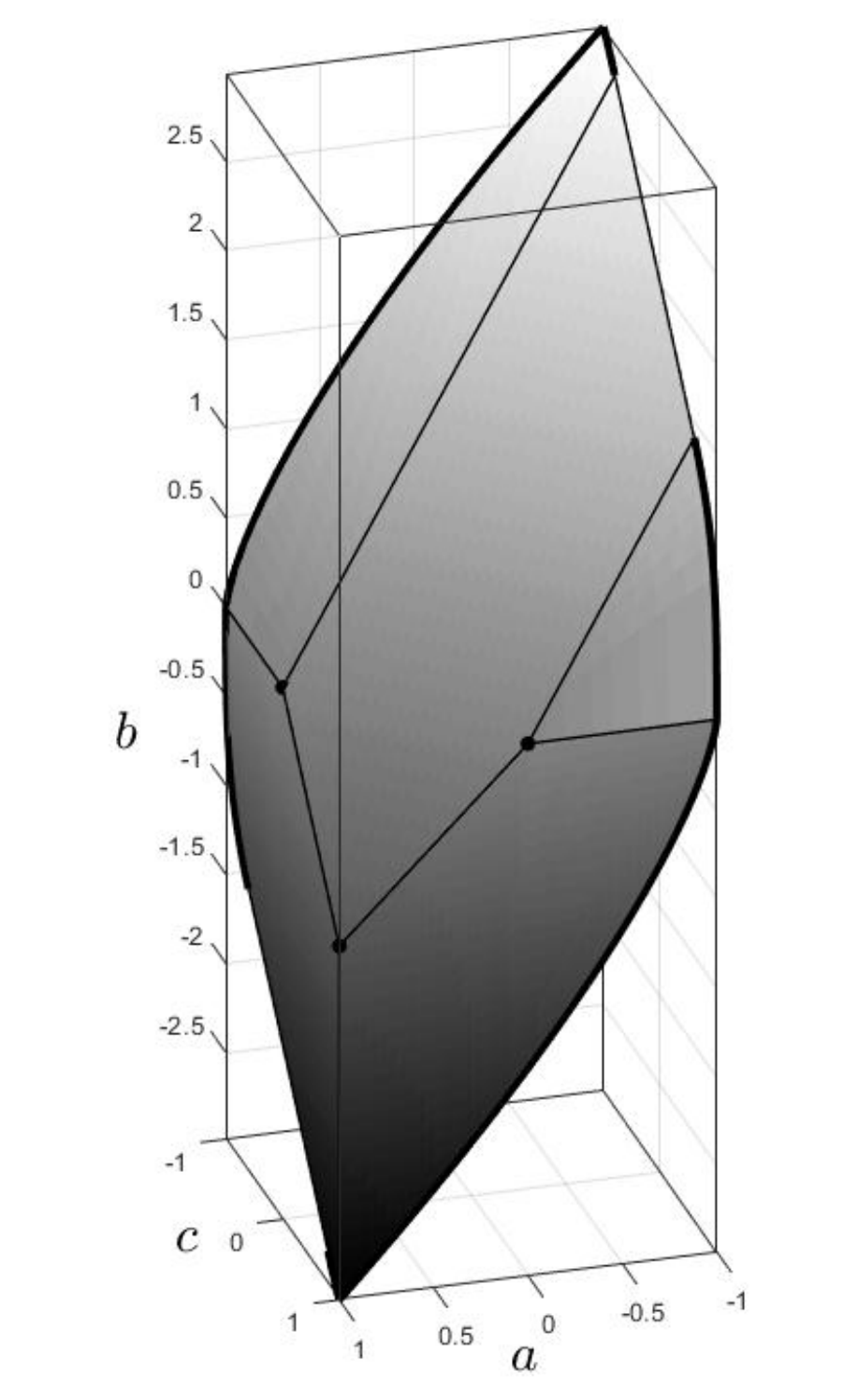}
	\caption{Representation of ${\mathsf S}_{m,n}^h$ in Case \ref{caseB} for $\frac{n}{m}<\frac{1}{3}$. Here we have considered the case $m=28$ and $n=8$. The extreme points have been drawn with a thicker line or isolated dots. The different surfaces that form ${\mathsf S}_{m,n}^h$ are delimited by thin lines. }\label{fig:BallnmLess13}
\end{figure}

\begin{figure}
	\centering
	\includegraphics[height=.7\textwidth,keepaspectratio=true]{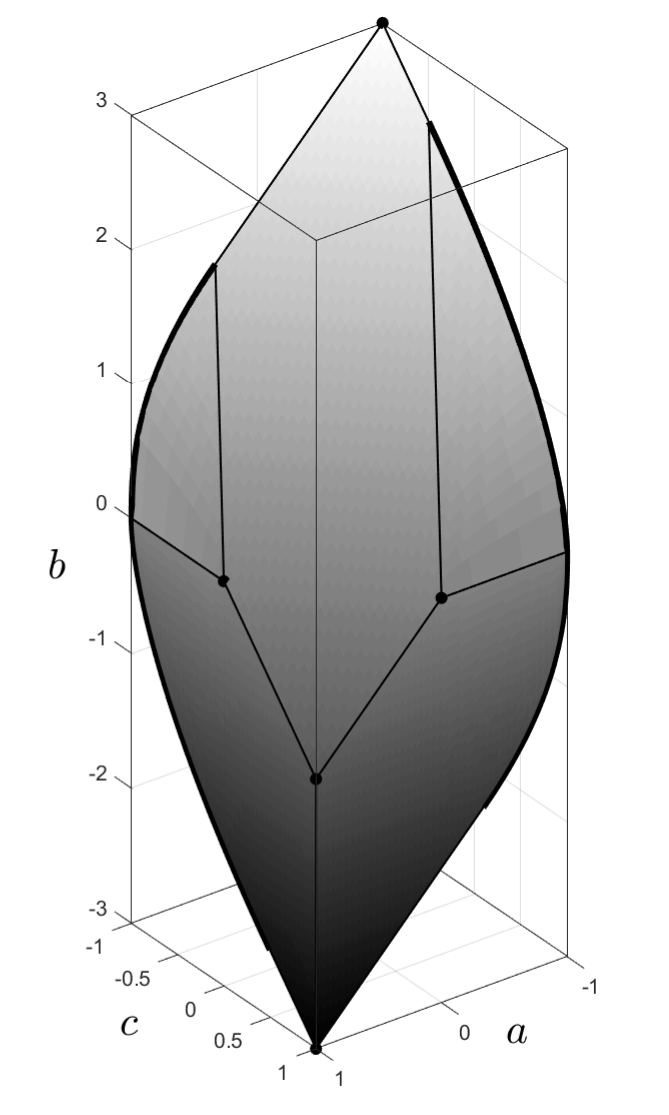}
	\caption{Representation of ${\mathsf S}_{m,n}^h$ in Case \ref{caseB} for $\frac{1}{3}\leq \frac{n}{m}\leq\frac{2}{3}$. Here we have considered the case $m=20$ and $n=12$. The extreme points have been drawn with a thicker line or isolated dots. The different surfaces that form ${\mathsf S}_{m,n}^h$ are delimited by thin lines. }\label{fig:BallnmBetween13and23}
\end{figure}

Of course, the description of the extreme points of the corresponding unit sphere is also hard to tackle; however, this was successfully achieved in \cite{GMS2023}. 

\begin{proposition}[\mbox{\cite[Theorem 4.2]{GMS2023}}]
	Let $m,n\in{\mathbb N}$ be even numbers with $m>n$. Also, as earlier, we let $L_{m,n}=\frac{m}{m-n}\left(\frac{m-n}{n}\right)^\frac{n}{m}$, $\lambda_0=-\frac{n}{m-n}$ and 
	$
	R_{m,n}=2^\frac{m-n}{m} L_{m,n}.
	$
	The set of extreme points $\ext({\mathsf B}_{m,n}^h)$ is given by
	\begin{enumerate}
		\item If $\frac{n}{m}\in\left(0,\frac{1}{3}\right)$ then
		\begin{align*}
			\ext({\mathsf B}_{m,n}^h)&=\left\{\pm\left(-1,L_{m,n}(1-c)^\frac{n}{m},c\right):c\in[1+\lambda_0,1]\right\}\\
			&\quad\quad \bigcup \left\{\pm\left(-1,R_{m,n}|c|^\frac{n}{m},c\right):c\in[-1,2\lambda_0]\right\}\\
			&\quad\quad \bigcup\left\{\pm\left(a,L_{m,n}(1-a)^\frac{m-n}{m},-1\right):a\in[-1,1]\right\}\\
			&\quad\quad \bigcup \left\{\pm(0,0,1),\pm(1,0,0), \pm (1,-1,1)\right\}.
		\end{align*}
		
		\item If $\frac{n}{m}\in\left[\frac{1}{3},\frac{2}{3}\right]$ then
		\begin{align*}
			\ext({\mathsf B}_{m,n}^h)&=\left\{\pm\left(-1,L_{m,n}(1-c)^\frac{n}{m},c\right):c\in[1+\lambda_0,1]\right\}\\
			&\quad\quad \bigcup\left\{\pm\left(a,L_{m,n}(1-a)^\frac{m-n}{m},-1\right):a\in\left[1+\frac{1}{\lambda_0},1\right]\right\}\\
			&\quad\quad \bigcup \left\{\pm(0,0,1),\pm(1,0,0), \pm (1,-1,1), \pm (1,-3,1)\right\}.
		\end{align*}
		
		\item If $\frac{n}{m}\in \left(\frac{2}{3},1\right)$ then
		\begin{align*}
			\ext({\mathsf B}_{m,n}^h)&=\left\{\pm\left(-1,L_{m,n}(1-c)^\frac{n}{m},c\right):c\in[-1,1]\right\}\\
			&\quad\quad \bigcup \left\{\pm\left(a,R_{m,m-n}|a|^\frac{m-n}{m},-1\right):a\in\left[-1,\frac{2}{\lambda_0}\right]\right\}\\
			&\quad\quad \bigcup\left\{\pm\left(a,L_{m,n}(1-a)^\frac{m-n}{m},-1\right):a\in\left[1+\frac{1}{\lambda_0},1\right]\right\}\\
			&\quad\quad \bigcup \left\{\pm(0,0,1),\pm(1,0,0), \pm (1,-1,1)\right\}.
		\end{align*}
	\end{enumerate}
\end{proposition}

Some sketches of the unit spheres together with the corresponding sets of extreme points are shown in Figures \ref{fig:BallnmLess13} and \ref{fig:BallnmBetween13and23}.

Certainly, as the reader can see, Case \ref{caseB} is way more complex than Case \ref{caseA}, at least at first sight. In fact, delving into the details and technicalities of the constructions from \cite{JMR2021} or \cite{GMS2023} (see also the recent monograph \cite{FGMMRS}), one realizes that these types of problems are far from easy to tackle.


\subsection{Outline of our paper}\mbox{ }\newline
\indent So far, we have briefly discussed the formulas for $\vertiii{\cdot}_{m,n}$ and the complete descriptions of the extreme points of ${\mathsf B}_{m,n}^h$ in Case \ref{caseA} and Case \ref{caseB}. 
One of the most challenging aspects of solving Case \ref{caseC} is that many of the equations that naturally arise in the computation are implicit, unlike the explicit equations encountered in Case \ref{caseA} or Case \ref{caseB}. One way to address this difficulty is by implementing an appropriate change of variables, which requires considering two different coordinate systems simultaneously. Despite these difficulties, the rest of the paper will be devoted to solving Case \ref{caseC}. From now on, unless otherwise specified, $m$ and $n$ will be positive integers such that $m > n$, with $m$ being even and $n$ being odd. Solving this final case will complete the general analysis, covering all possibilities for $m$ and $n$.

The outline of the rest of this paper is as follows: Section \ref{sec:formula} is devoted to obtain an explicit formula for $\vertiii{\cdot}_{m,n}$ for all positive integers $m>n$ with $m$ even and $n$ odd. The formula obtained in Section \ref{sec:formula} will be used in Section \ref{sec:projection} to calculate the projection of ${\mathsf B}^h_{m,n}$ over the plane $ac$ together with a parametrization of ${\mathsf S}^h_{m,n}$. Thanks to the results in Sections \ref{sec:formula} and \ref{sec:projection}, we will be able to complete this thorough study by providing the extreme points of ${\mathsf B}^h_{m,n}$ in Section \ref{sec:extreme}; thus resolving Case \ref{caseC}. Additionally, the sphere ${\mathsf S}^h_{m,n}$ will be sketched for several choices of $m$ and $n$.

The following notations will be useful to understand the rest of the paper. If $C$ is a convex body, $\ext(C)$ will denote the set of extreme points of $C$. Also,
$\pi_{ac}$ will denote the linear projection given by
$\pi_{ac}(a,b,c)=(a,c)$, for every $(a,b,c)\in{\mathbb R}^3$. The 
plots of ${\mathsf S}^h_{m,n}$ and the projection $\pi_{ac}({\mathsf S}^h_{m,n})$, together with some other figures appearing in this paper were generated
using {\em MATLAB}. All graphs presented here are scaled.

\section{A formula for $\vertiii{\cdot}_{m,n}$ with $m$ even and $n$ odd}\label{sec:formula}

Given $m, n \in \mathbb{N}$ with $m >n$ and $(a,b,c) \in \mathbb{R}^3$, recall from \eqref{eq:relation} and \eqref{eq:reduction} that
\begin{equation*}
	\vertiii{(a,b,c)}_{m,n}=\max\{\|(a,b,c)\|_{m,m-n},\|(c,b,a)\|_{m,n}\},
\end{equation*}
and
\[
\vertiii{(a,b,c)}_{m,n}=\vertiii{(c,b,a)}_{m,m-n}.
\]
The second identity allows us to simplify some of the forthcoming proofs in the sense that it will be enough to consider the case where $\frac{n}{m}\leq \frac{1}{2}$.

In \cite{MS}, the authors derived a formula to calculate $\|(a,b,c)\|_{m,m-n}$ and $\|(c,b,a)\|_{m,n}$, which we will state for completeness in the following result:
\begin{theorem}[\mbox{Mu\~noz and Seoane, \cite[Theorem 4.1]{MS}}]\label{evenEvenTheorem}
	For every $m,n\in{\mathbb N}$ with $m$ even, $n$ odd and $m>n$, let
	us define ${\mathcal I}_{m,n}$ as the set of triples
	$(a,b,c)\in\mathbb{R}^3$ such that
	$$
	a\ne 0,\ \left|\frac{nb}{ma}\right|< 1\text{ and } 1+\frac{c}{a}<\frac{1}{2}\left[\frac{m-n}{n}\left(\left|\frac{nb}{ma}\right|\right)^\frac{m}{m-n}-\left|\frac{b}{a}\right|+1\right].
	$$
	Then we have
	\begin{equation}\label{formulaEvenEven}
		\|(a,b,c)\|_{m,n}=\begin{cases}
			\left| \frac{(m-n)a}{n} \left(\left|\frac{nb}{ma}\right|\right)^{\frac{m}{m-n}} - c
			\right|&\text{if $(a,b,c) \in \mathcal{I}_{m,n}$,}\\
			\left|a+c\right|+\left|b\right|&\text{otherwise.}
		\end{cases}
	\end{equation}
\end{theorem}

In the first main result of this paper, we will derive a formula to calculate $\vertiii{(a,b,c)}_{m,n}$ for the case where $m$ is even, $n$ is odd, and $m \ge 2n$. We will introduce necessary definitions and auxiliary results.

\begin{lemma}\label{lem:AmnBmn}
	Let $m$ and $n$ be positive integers such that $m$ is even, $n$ is odd and $m\ge 2n$. Consider the regions $\mathcal{I}_{m,n}$ and $\mathcal{I}_{m,m-n}$ from Theorem \ref{evenEvenTheorem}. Define
	\begin{align*}
	{\mathbb A}_{m,n}&=\{(b,c)\in{\mathbb R}^2:(1,b,c)\in{\mathcal I}_{m,m-n}\},\\
	{\mathbb B}_{m,n}&=\{(b,c)\in{\mathbb R}^2:(c,b,1)\in{\mathcal I}_{m,n}\}.
	\end{align*}
	Then 
	\begin{enumerate}
		\item ${\mathbb A}_{m,n}=\left\{(\pm b,c):c< -1/2,\ 0\leq b< \frac{m}{m-n}\text{ and } f(b)\leq \frac{nb}{mc}\right\}$ and
		\item ${\mathbb B}_{m,n}=\left\{(\pm b,c):c\ne 0,\ -1< \frac{nb}{mc}\leq 0\text{ and } 0\leq b \leq g\left(\frac{nb}{mc}\right)\right\}$,
	\end{enumerate}
	where
	\begin{align*}
	f(b)&=\frac{2nb}{mK_{m,n}b^\frac{m}{n}-mb-m},\\
	g(t)&=\frac{2mt}{(m-n)t^\frac{m}{m-n} +mt-n}
	\end{align*}
	and $K_{m,n}=\frac{n}{m-n}\left(\frac{m-n}{m}\right)^\frac{m}{n}$.
\end{lemma}

\begin{proof}
Let us prove first $(1)$. Using the definition of ${\mathcal I}_{m,m-n}$, we have that $(b,c)\in{\mathbb A}_{m,n}$ if and only if
	$$
	|b|<\frac{m}{m-n}
	$$
and
	$$
	c<\frac{1}{2}\left[K_{m,n}|b|^\frac{m}{n}-|b|-1\right].
	$$
Let $h_1(b)=\frac{1}{2}\left[K_{m,n}|b|^\frac{m}{n}-|b|-1\right]$ for $b\in[-\frac{m}{m-n},\frac{m}{m-n}]$. Since $h_1$ is even, we just need to investigate $h_1$ in $[0,\frac{m}{m-n}]$, where $f$ is given by $h_1(b)=\frac{1}{2}\left[K_{m,n}b^\frac{m}{n}-b-1\right]$. It is elementary to show that $h_1'(b)=\frac{m}{n}K_{m,n}b^\frac{m-n}{n}-1$ vanishes only at $b=\frac{m}{m-n}$ and that $h_1''(\frac{m}{m-n})>0$. Hence $h_1$ has its absolute miminum at $b=\frac{m}{m-n}$ and therefore $h_1$ is strictly decreasing in $[0,\frac{m}{m-n}]$. Consequently $h_1$ attains it maximun at $b=0$ and so $h_1(b)\leq h(0)=-1/2$ for all $b\in[-\frac{m}{m-n},\frac{m}{m-n}]$. Now observe that $(b,c)\in {\mathbb A}_{m,n}$ if and only if $(-b,c)\in {\mathbb A}_{m,n}$. This allows us to focus on the pairs $(b,c)$ with non negative $b$. Let us take $b\ge 0$. Assume first that $b>0$. Then $(b,c)\in{\mathbb A}_{m,n}$ is equivalent to
	$$
	0<b<\frac{m}{m-n}\quad\text{and}\quad c<\frac{1}{2}\left[K_{m,n}|b|^\frac{m}{n}-|b|-1\right]\leq-\frac{1}{2}.
	$$
On the other hand, the inequality $c<\frac{1}{2}\left[K_{m,n}|b|^\frac{m}{n}-|b|-1\right]$ is equivalent to
	$$
	\frac{2}{K_{m,n}|b|^\frac{m}{n}-|b|-1}<\frac{1}{c}.
	$$
Then multiplying the previous inequality by $\frac{nb}{m}$ (which is positive) we arrive at the equivalent condition
	$$
	\frac{2nb}{mK_{m,n}|b|^\frac{m}{n}-m|b|-m}<\frac{nb}{mc}
	$$
or 
	$$
	f(b)<\frac{nb}{mc}.
	$$
If $b=0$ it is straightforward that $(0,c)\in {\mathbb A}_{m,n}$ if and only if $c<h(0)=-1/2$. Also $f(0)=0$. Then we have shown that, for $b$ with $0\le b<\frac{m}{m-n}$, $(b,c)\in{\mathbb A}_{m,n}$ if and only if
	$$
	c< -\frac{1}{2}
	$$
and 
	$$
	f(b)\leq \frac{nb}{mc},
	$$
finishing the proof of $(1)$.

Now we prove $(2)$. From the definition of ${\mathcal I}_{m,n}$ (see Theorem \ref{evenEvenTheorem}), $(a,c)\in {\mathbb B}_{m,n}$ is equivalent to
	\begin{equation}\label{eq:Imn1}
	c\ne 0,\quad\left|\frac{b}{c}\right|<\frac{m}{n}
	\end{equation}
and
	\begin{equation}\label{eq:Imn2}
	\frac{1}{c}<\frac{1}{2}\left[K_{m,m-n}\left|\frac{b}{c}\right|^\frac{m}{m-n}-\left|\frac{b}{c}\right|-1\right].
	\end{equation}
Clearly, $(b,c)\in{\mathbb B}_{m,n}$ if and only if $(-b,c)\in{\mathbb B}_{m,n}$. In the rest of the proof we may assume then that $b\ge 0$. Assume first that $b>0$. Writing \eqref{eq:Imn1} and \eqref{eq:Imn2} in terms of $t=\frac{nb}{mc}$ we obtain
	$$
	c\ne 0,\quad \left|\frac{nb}{mc}\right|<1
	$$
and
	\begin{align*}
	\frac{m}{nb}\frac{nb}{mc}&<\frac{1}{2}\left[K_{m,m-n}\left(\frac{m}{n}\right)^\frac{m}{m-n}\left|\frac{nb}{mc}\right|^\frac{m}{m-n}-\frac{m}{n}\left|\frac{nb}{mc}\right|-1\right]\\
	&=\frac{1}{2}\left[\frac{m-n}{n}\left|\frac{nb}{mc}\right|^\frac{m}{m-n}-\frac{m}{n}\left|\frac{nb}{mc}\right|-1\right]
	\end{align*}
or equivalently
	$$
	c\ne 0,\quad\left|t\right|<1
	$$
	and
	\begin{equation}\label{eq:g}
	\frac{mt}{b}<\frac{1}{2}\left[(m-n)|t|^\frac{m}{m-n}-m|t|-n\right].
	\end{equation}
If $h_2(t)=\frac{1}{2}\left[(m-n)|t|^\frac{m}{m-n}-m|t|-n\right]$ for all $t$ with $|t|\leq 1$, then $h_2$ attains it absolute maximum at $t=0$. Indeed, we just need to notice that $h_2$ is even and that $h_2'(t)=\frac{m}{2}\left[ t^\frac{n}{m-n}-1\right]<0$ for all $t\in[0,1)$. Then $h_2(t)\leq h_2(0)=-\frac{n}{2}<0$ for all $t$ with $|t|\leq 1$. In particular, from \eqref{eq:g} it follows that $\frac{mt}{b}<0$, and since $b>0$, then $t\in(-1,0)$ Consequently, \eqref{eq:g} would be equivalent to
	$$
	b<\frac{2mt}{(m-n)|t|^\frac{m}{m-n}-m|t|-n}=g(t).
	$$
If $b=0$ then $t=0$ and $g(0)=0$, proving that $(b,c)\in {\mathbb B}_{m,n}$ with $b\ge 0$ is equivalent to 
	$$
	c\ne 0,\quad-1<t\leq 0
	$$
and 
	$$
	b\leq g(t).
	$$
This concludes the proof of $(2)$.
\end{proof}

In the following, on many occasions it will be useful to replace the variable $c$ by the variable $t$ with $t=\frac{nb}{mc}$. Doing so it is straightforward that $(b,c)\in {\mathbb A}_{m,n}$ (respectively $(b,c)\in {\mathbb B}_{m,n}$) if and only if $(b,t)\in A_{m,n}$ (respectively $(b,t)\in B_{m,n}$) where $A_{m,n}=A^1_{m,n}\cup A^2_{m,n}$ and $B_{m,n}={ B}^1_{m,n}\cup B^2_{m,n}$ with $A^2_{m,n}=-A^1_{m,n}$, $B^2_{m,n}=-B^1_{m,n}$ and
\begin{align*}
	A^1_{m,n}&=\left\{(b,t):\ 0\leq b< \frac{m}{m-n}\text{ and } f(b)\le t\leq 0\right\},\\
	B^1_{m,n}&=\left\{(b,t):\ -1< t\leq 0\text{ and } 0\leq b \leq g\left(t\right)\right\}.
\end{align*}
It is a good moment to take a look at the representation of $A^1_{m,n}$ and $B^1_{m,n}$ in Figure~\ref{fig:AB_intersection}.

\begin{lemma}\label{lem:Lambda}
	Let $m$ and $n$ be positive integers such that $m$ is even, $n$ is odd and $m\ge 2n$. If $K_{m,n}=\frac{n}{m-n}\left(\frac{m-n}{m}\right)^\frac{m}{n}$, the equation 
	\begin{equation}\label{eq:equality1}
	mK_{m,n}tb^\frac{m}{n}-nb-mt+(m-n)bt^\frac{m}{m-n}=0,
	\end{equation}
	defines an strictly decreasing function $t=\Lambda_{m,n}(b)$ defined for $b\in \left[0, \frac{m}{m-n}\right]$ such that $\Lambda_{m,n}(0)=0$. We define $\tau_0=\Lambda_{m,n}\left(\frac{m}{m-n}\right)$.
\end{lemma}

\begin{proof}
	If $\Lambda_{m,n}$ exists then it is straightforward that $\Lambda_{m,n}(0)=0$. Now, differentiating \eqref{eq:equality1} with respect to $b$ we arrive at
	$$
	t'=F(b,t)=\frac{n^2-m^2K_{m,n}tb^\frac{m-n}{n}-n(m-n)t^\frac{m}{m-n}}{mn\left[K_{m,n}b^\frac{m}{n}+bt^\frac{n}{m-n}-1\right]}.
	$$
	Hence, if $\Lambda_{m,n}$ exists, it must be a solution to the initial value problem
	\begin{equation}\label{eq:IVP}
	\begin{cases}
	&t'=F(b,t),\\
	&t(0)=0.
	\end{cases}
	\end{equation}
	Now consider the open subset of ${\mathbb R}^2$ given by
	$$
	{\mathcal U}=\{(b,t)\in{\mathbb R}^2:K_{m,n}b^\frac{m}{n}+bt^\frac{n}{m-n}-1<0\}.
	$$
	The set ${\mathcal U}$ has been represented in Figure~\ref{fig:U}. Observe that $(0,0)\in{\mathcal U}$. Then, since $\frac{\partial F}{\partial b}$ and $\frac{\partial F}{\partial t}$ are both continuous in ${\mathcal U}$, the initial value problem \eqref{eq:IVP} has a unique solution that can be extended to $\partial {\mathcal U}$.
	\begin{figure}
		\centering
		\includegraphics[height=.65\textwidth,keepaspectratio=true]{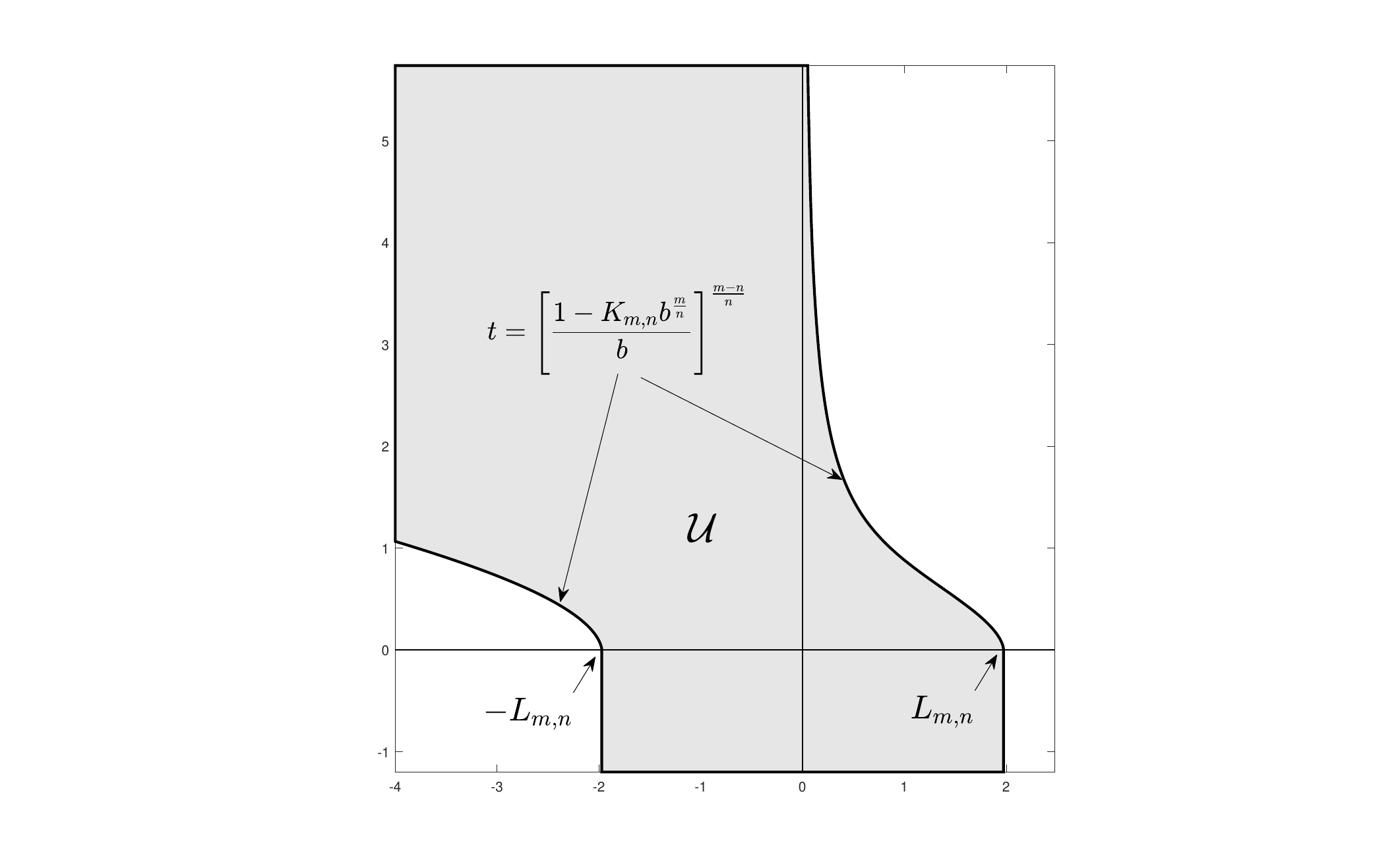}
		\caption{Representation of ${\mathcal U}$. Here we have considered the case $m=12$ and $n=5$. }\label{fig:U}
	\end{figure}
	On the other hand, dividing \eqref{eq:equality1} by $bt$ (assuming $bt\ne 0$),
	$$
	mK_{m,n}b^\frac{m-n}{n}-\frac{n}{t}-\frac{m}{b}+(m-n)t^\frac{n}{m-n}=0.
	$$
	Differentiating again with respect to $b$,
	$$
	t'=\frac{-mn-m(m-n)K_{m,n}b^\frac{m}{n}}{n^2t^\frac{m}{m-n}+n^2}\frac{t^2}{b^2},
	$$
	which is clearly negative for every plausible value of $b$ and $t$. Hence $\Lambda_{m,n}$ is strictly decreasing and its domain contains $(-L_{m,n},L_{m,n})$ and therefore $\left[-\frac{m}{m-n},\frac{m}{m-n}\right]$ too.
\end{proof}

\begin{remark}
The value of $\tau_0$ for every plausible choice of $m$ and $n$ is the unique solution between $-1$ and $0$ of the equation
	\begin{equation}\label{eq:tau_0}
	(m-n)t^\frac{m}{m-n}+(2n-m)t-n=0.
	\end{equation}
The equation \eqref{eq:tau_0} cannot generally be solved explicitly. However, using specialized software $\tau_0$ can be approximated with precision. We provide below a table with some approximations to $\tau_0$ obtained with \textit{Mathematica}: 
\begin{center}%
	\begin{tabular}[c]{|c|c|c|}%
		\hline
		$m$ & $n$ & $\tau_0$\\
		\hline
		$4$ & $1$ & $-0.2560771804$ \\
		
		$6$ & $1$ & $-0.1359670417$ \\
		
		$8$ & $1$ & $-0.0911451357$ \\
		
		$8$ & $3$ & $-0.5472162244$ \\
		
		$10$ & $1$ & $-0.0681314528$ \\
		
		$10$ & $3$ & $-0.3536273979$ \\
		
		$12$ & $1$ & $-0.0542309739$ \\
		
		$12$ & $3$ & $-0.2560771804$ \\
		
		$12$ & $5$ & $-0.6823509843$\\
		\hline
	\end{tabular}	
\end{center}

For $m=4$ and $n=1$ (or $\frac{m}{n} = 4$ in any case) it can be seen that $\tau_0$ is exactly equal to 
	$$
	\tau_0 = -\frac{1}{81} \left(\sqrt[3]{729 \sqrt{17}+541}-\frac{206}{\sqrt[3]{729 \sqrt{17}+541}}+19\right).
	$$
	To give the reader an idea of how complex these calculations can get, for low values of $m$ and $n$, $\tau_0$ is a real root of a polynomial. For instance, if $m=8$ and $n=3$, $\tau_0$ is the only real root of the $7$ degree polynomial given by
	$$
	3125 t^7+3125 t^6+3125 t^5+3093 t^4+2853 t^3+2133 t^2+1053 t+243.
	$$
\end{remark}

\begin{theorem}
	Let $m$ and $n$ be positive integers such that $m$ is even, $n$ is odd and $m>n$. Let us define $K_{m,n}=\frac{n}{m-n}\left(\frac{m-n}{m}\right)^\frac{m}{n}$. If $m\geq 2n$ then
	\begin{equation}\label{eq:norm}
	\vertiii{(a,b,c)}_{m,n}=\begin{cases}
	\left|K_{m,n}a\left|\frac{b}{a}\right|^\frac{m}{n}-c\right|&\text{if $abc\ne 0$ and $\left(\frac{b}{a},\frac{nb}{mc}\right)\in {\mathcal A}_{m,n}$},\\
	\left|K_{m,m-n}c\left|\frac{b}{c}\right|^\frac{m}{m-n}-a\right|&\text{if $abc\ne 0$ and $\left(\frac{b}{a},\frac{nb}{mc}\right)\in {\mathcal B}_{m,n}$},\\
	\max\{|a|,|c|\}&\text{if $b=0$ and $ac\leq 0$},\\
	|a+c|+|b|&\text{otherwise},
	\end{cases}
	\end{equation}
	where ${\mathcal A}_{m,n}={\mathcal A}^1_{m,n}\cup {\mathcal A}^2_{m,n}$, ${\mathcal B}_{m,n}={\mathcal B}^1_{m,n}\cup {\mathcal B}^2_{m,n}$, 
	\begin{align*}
	{\mathcal A}^1_{m,n}&=\{(b,t)\in{\mathbb R}^2:b\in(0,\frac{m}{m-n}]\text{ and } \Lambda_{m,n}(b)\leq t< 0\},\\
	{\mathcal B}^1_{m,n}&=\{(b,t)\in{\mathbb R}^2:b\in(0,\frac{m}{m-n}]\text{ and } \tau_0\leq t\leq \Lambda_{m,n}(b)\}\\
	&\quad\quad\cup \{(b,t)\in{\mathbb R}^2:t\in[-1,\tau_0]\text{ and } 0< b\leq g(t)\}.
	\end{align*}
	and ${\mathcal A}^2_{m,n}=-{\mathcal A}^1_{m,n}$, ${\mathcal B}^2_{m,n}=-{\mathcal B}^1_{m,n}$. 
	The regions ${\mathcal A}^1_{m,n}$ and ${\mathcal B}^1_{m,n}$ have been represented in Figure~\ref{fig:AB}. On the other hand,
	if $m<2n$ then, according to \eqref{eq:reduction}, $\vertiii{(a,b,c)}_{m,n}=\vertiii{(c,b,a)}_{m,m-n}$.
\end{theorem}

\begin{proof}
	It is elementary to prove that
	\begin{align*}
	\vertiii{(0,b,c)}_{m,n}&=|b|+|c|,\\
	\vertiii{(a,b,0)}_{m,n}&=|a|+|b|,
	\end{align*}
	for all $a,b,c\in{\mathbb R}$. Also,
	\begin{align}
	\vertiii{(a,0,c)}_{m,n}&=\max\{|ax^m+cy^m|:(x,y)\in[-1,1]^2\}\nonumber\\
	&=\begin{cases}
	\max\{|a|,|c|\} & \text{if $ac\leq 0$},\\
	|a+c| & \text{if $ac> 0$}.\label{align:bequalszero}\\
	\end{cases}
	\end{align}
	For the rest of the proof we will assume that $a,b,c\ne 0$. In particular we have
	$$
	\vertiii{(a,b,c)}_{m,n}=|a|\cdot\vertiii{\left(1,\frac{b}{a},\frac{c}{a} \right)}_{m,n}
	$$
	and therefore it will be enough to obtain a formula to calculate $\vertiii{(1,b,c)}_{m,n}$ for all $b,c\in{\mathbb R}$ with $b,c\ne 0$.
	
	According to \eqref{eq:relation} we have
	$$
	\vertiii{(1,b,c)}_{m,n}=\max\{\|(1,b,c)\|_{m,m-n},\|(c,b,1)\|_{m,n}\}.
	$$
	Now using \eqref{formulaEvenEven} we obtain	
	\begin{align*}
	\|(1,b,c)\|_{m,m-n}&=
	\begin{cases}
	\left|K_{m,n}|b|^\frac{m}{n}-c\right|&\text{if $(1,b,c)\in{\mathcal I}_{m,m-n}$,}\\
	|1+c|+|b|&\text{otherwise,}
	\end{cases}
	\\
	\|(c,b,1)\|_{m,n} &=
	\begin{cases}
	\left|K_{m,m-n}c\left|\frac{b}{c}\right|^\frac{m}{m-n}-1\right|&\text{if $(c,b,1)\in{\mathcal I}_{m,n}$,}\\
	|1+c|+|b|&\text{otherwise.}
	\end{cases}	
	\end{align*}
	Since $\|(1,b,c)\|_{m,m-n}=\|(1,-b,c)\|_{m,m-n}$ and $\|(c,b,1)\|_{m,n}=\|(c,-b,1)\|_{m,n}$ we can assume that $b> 0$. From Lemma~\ref{lem:AmnBmn} we have:
	\begin{itemize}
		\item $(1,b,c)\in{\mathcal I}_{m,m-n}$ is equivalent to $(b,t)\in A_{m,n}$.
		
		\item $(c,b,1)\in{\mathcal I}_{m,m}$ is equivalent to $(b,t)\in B_{m,n}$.
	\end{itemize}
As a matter of fact, since $b> 0$ in fact we have
	\begin{itemize}
	\item $(1,b,c)\in{\mathcal I}_{m,m-n}$ is equivalent to $(b,t)\in A_{m,n}^1$.
	
	\item $(c,b,1)\in{\mathcal I}_{m,m}$ is equivalent to $(b,t)\in B_{m,n}^1$.
\end{itemize}
Therefore, if $b> 0$ we have that
	$$
	\vertiii{(1,b,c)}_{m,n}=\begin{cases}
	\left|K_{m,n}|b|^\frac{m}{n}-c\right|&\text{if $(b,t)\in A^1_{m,n}\setminus B^1_{m,n}$,}\\
	\left|K_{m,m-n}c\left|\frac{b}{c}\right|^\frac{m}{m-n}-1\right|&\text{if $(b,t)\in B^1_{m,n}\setminus A^1_{m,n}$},\\
	|1+c|+|b|&\text{if $(b,t)\notin A^1_{m,n}\cup B_{m,n}^1$},
	\end{cases}
	$$	 
and
	$$
	\vertiii{(1,b,c)}_{m,n}=\max\left\{\left|K_{m,n}|b|^\frac{m}{n}-c\right|,\left|K_{m,m-n}c\left|\frac{b}{c}\right|^\frac{m}{m-n}-1\right|\right\}
	$$
whenever $(b,t)\in A^1_{m,n}\cap B_{m,n}^1$. To finish the proof we just need to compare $\left|K_{m,n}|b|^\frac{m}{n}-c\right|$ and $\left|K_{m,m-n}c\left|\frac{b}{c}\right|^\frac{m}{m-n}-1\right|$ within $A^1_{m,n}\cap B_{m,n}^1$.

It is left to the reader to check that the inequality
	$$
	\left|K_{m,n}|b|^\frac{m}{n}-c\right|\ge\left|K_{m,m-n}c\left|\frac{b}{c}\right|^\frac{m}{m-n}-1\right|
	$$
can be written using the variables $b$ and $t$ as
	\begin{equation}\label{eq_FG_bt}
	\left|K_{m,n}b^\frac{m}{n}-\frac{nb}{mt}\right|\ge\left|\frac{m-n}{m}bt^\frac{n}{m-n}-1\right|.
	\end{equation}
Now, let $F(b,t)$ and $G(b,t)$ be, respectively, the right and the left hand side of~\eqref{eq_FG_bt}. 
Notice that
	$$
	F(b,t)=1-\frac{m-n}{m}bt^\frac{n}{m-n}
	$$
for every $(b,t)\in A^1_{m,n}$ and
	$$
	G(b,t)=K_{m,n}b^\frac{m}{n}-\frac{nb}{mt}
	$$
for all $(b,t)\in B^1_{m,n}$ with $t\ne 0$. The equality 
	$$
	G(t,b)=F(t,b)
	$$
is equivalent to
	$$
	mK_{m,n}tb^\frac{m}{n}-nb-mt+(m-n)bt^\frac{m}{m-n}=0
	$$
which is the identity \eqref{eq:equality1} studied in Lemma~\ref{lem:Lambda}. Hence $F$ and $G$ coincide only along the curve $t=\Lambda_{m,n}(b)$. Interestingly, $t=\Lambda_{m,n}(b)$ and $b=g(t)$ meet at the point $\left(\frac{m}{m-n},\tau_0\right)$. Indeed, from Lemma \ref{lem:Lambda}, keeping in mind that $K_{m,n}=\frac{n}{m-n}\left(\frac{m-n}{m}\right)^\frac{m}{n}$, and making
$$t = \tau_0 \text{ and } b = \frac{m}{m-n},$$
which entails that $\Lambda_{m,n}(b) = \tau_0$, we have the identity
\begin{align*}
0 & = \frac{m n}{m-n}\left(\frac{m-n}{m}\right)^\frac{m}{n} \tau_0 \left( \frac{m}{m-n} \right)^{\frac{m}{n}} - \frac{m n}{m-n} - m \tau_0 + \frac{(m-n) m}{m-n} \tau_0^{\frac{m}{m-n} }.	
\end{align*}
After some straightforward calculations we arrive at
\begin{equation}\label{tau_x}
(m-n) \tau_0^{ \frac{m}{m-n} }	= n + (m-n) \tau_0 - n \tau_0.
\end{equation}
Next, let us see that $$g(\tau_0) = \frac{m}{m-n}.$$ 
Indeed, using equation \eqref{tau_x}, we have
\begin{align*}
g(\tau_0) & = \frac{2 m \tau_0}{(m-n) \tau_0^{ \frac{m}{m-n} }	+ m \tau_0 - n} \\
& = \frac{2 m \tau_0}{n + (m-n) \tau_0 - n \tau_0	+ m \tau_0 - n} \\
& = \frac{m}{m-n}.
\end{align*}
Also, let us notice that equation \eqref{tau_x} has only one solution in the interval $(-1,0)$. Indeed, using a convexity argument, it is clear that the functions $(m-n) t^{\frac{m}{m-n}}$ and $n + (m-n) t - n t$ meet, at most, twice. Since $t=1$ makes $(m-n) t^{\frac{m}{m-n}} = n + (m-n) t - n t$ and we know that there is one solution belonging to $(-1,0)$, the above $\tau_0$ in unique.

We conclude that $F(t,b)\ge G(t,b)$ if $(b,t)\in A^1_{m,n}\cup B_{m,n}^1$ and $(b,t)$ is below the curve $t=\Lambda_{m,n}(b)$ or in the curve, or equivalently in ${\mathcal B}^1_{m,n}$. On the contrary $F(t,b)\leq G(t,b)$ in the rest of $A^1_{m,n}\cup B_{m,n}^1$ or along the curve $t=\Lambda_{m,n}(b)$, that is, in ${\mathcal A}^1_{m,n}$. Then
	$$
	\vertiii{(1,b,c)}_{m,n}=\begin{cases}
	\left|K_{m,n}|b|^\frac{m}{n}-c\right|&\text{if $(b,t)\in {\mathcal A}^1_{m,n}$,}\\
	\left|K_{m,m-n}c\left|\frac{b}{c}\right|^\frac{m}{m-n}-1\right|&\text{if $(b,t)\in {\mathcal B}^1_{m,n}$},\\
	|1+c|+|b|&\text{if $(b,t)\notin {\mathcal A}^1_{m,n}\cup {\mathcal B}_{m,n}^1$}.
	\end{cases}
	$$	
We arrive at the desired result by combining the previous formula with
	$$
	\vertiii{(a,b,c)}_{m,n}=|a|\cdot\vertiii{\left(1,\frac{b}{a},\frac{c}{a} \right)}_{m,n}.
	$$ 
\end{proof}

\begin{figure}
	\centering
	\includegraphics[height=.5\textwidth,keepaspectratio=true]{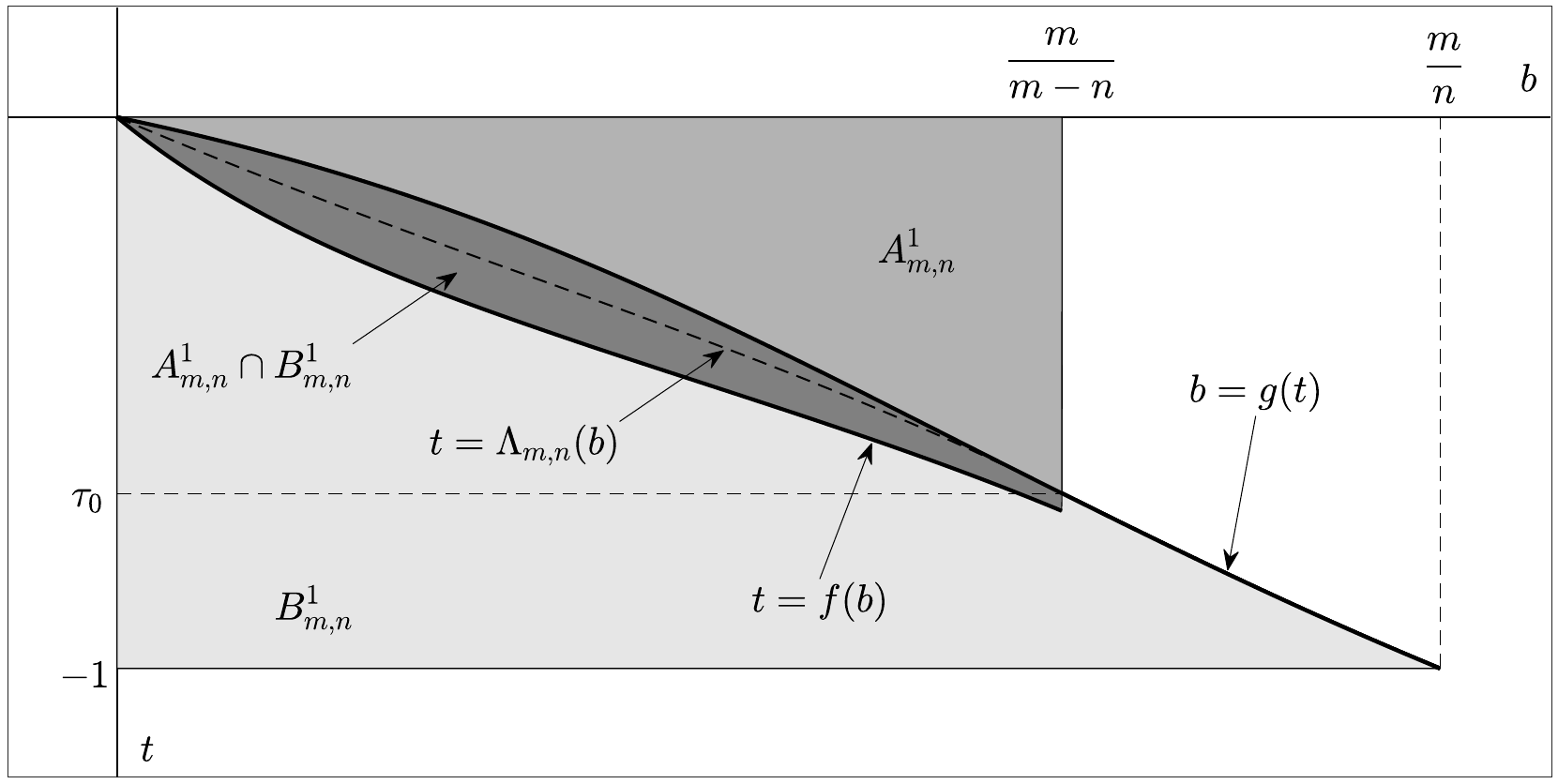}
	\caption{Representation of $A_{m,n}^1$ and $B_{m,n}^1$. The intersection $A_{m,n}^1\cap B_{m,n}^1$ is the darkest region. We have chosen the values $m=12$ and $n=5$. }\label{fig:AB_intersection}
\end{figure}

\begin{figure}
	\centering
	\includegraphics[height=.5\textwidth,keepaspectratio=true]{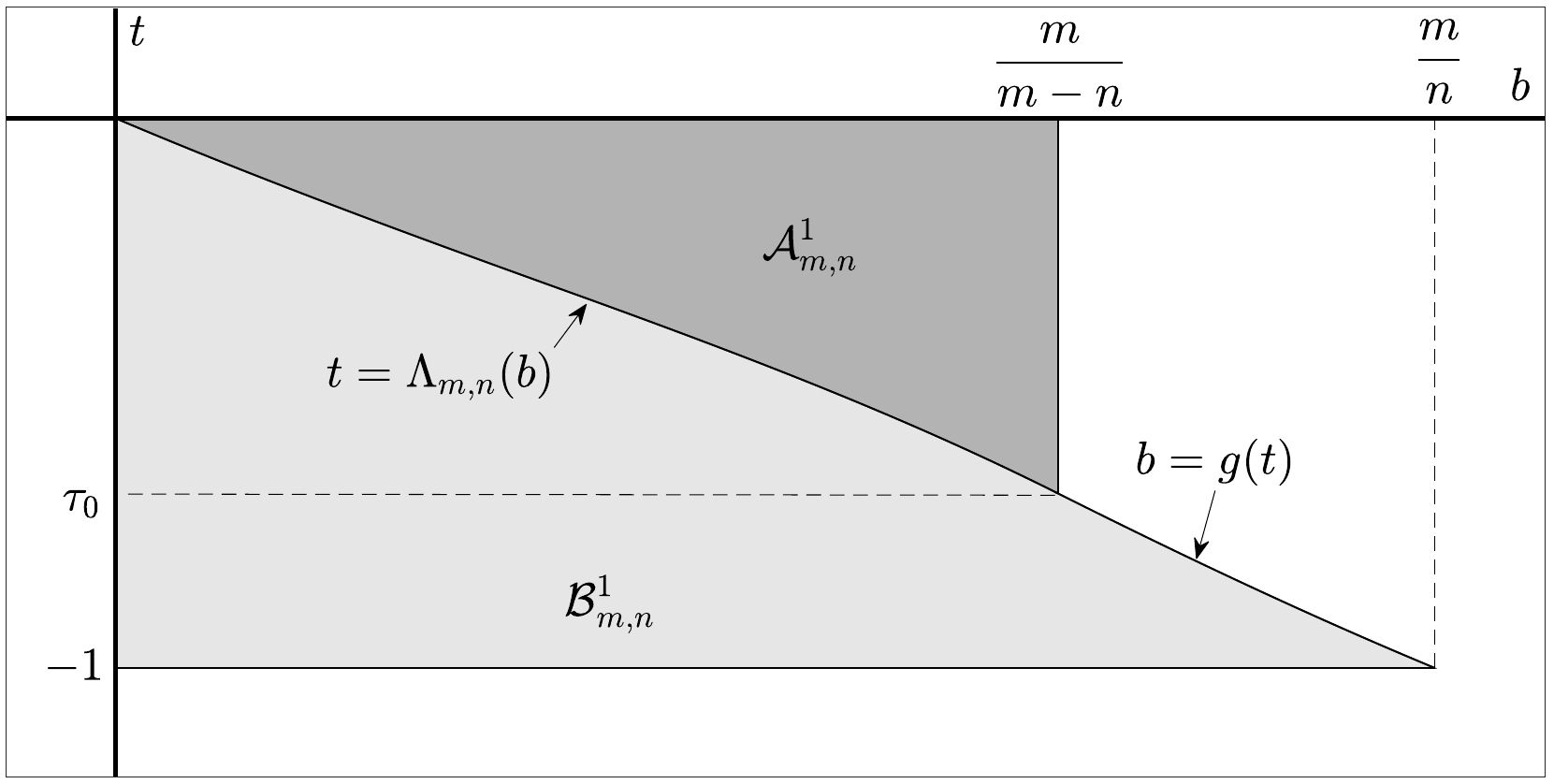}
	\caption{Representation of ${\mathcal A}_{m,n}^1$ and ${\mathcal B}_{m,n}^1$. Here we have considered the case $m=12$ and $n=5$. }\label{fig:AB}
\end{figure}

\section{A parametrization of the unit sphere ${\mathsf S}^h_{m,n}$}\label{sec:projection}

\begin{figure}
	\centering
	\includegraphics[height=.75\textwidth,keepaspectratio=true]{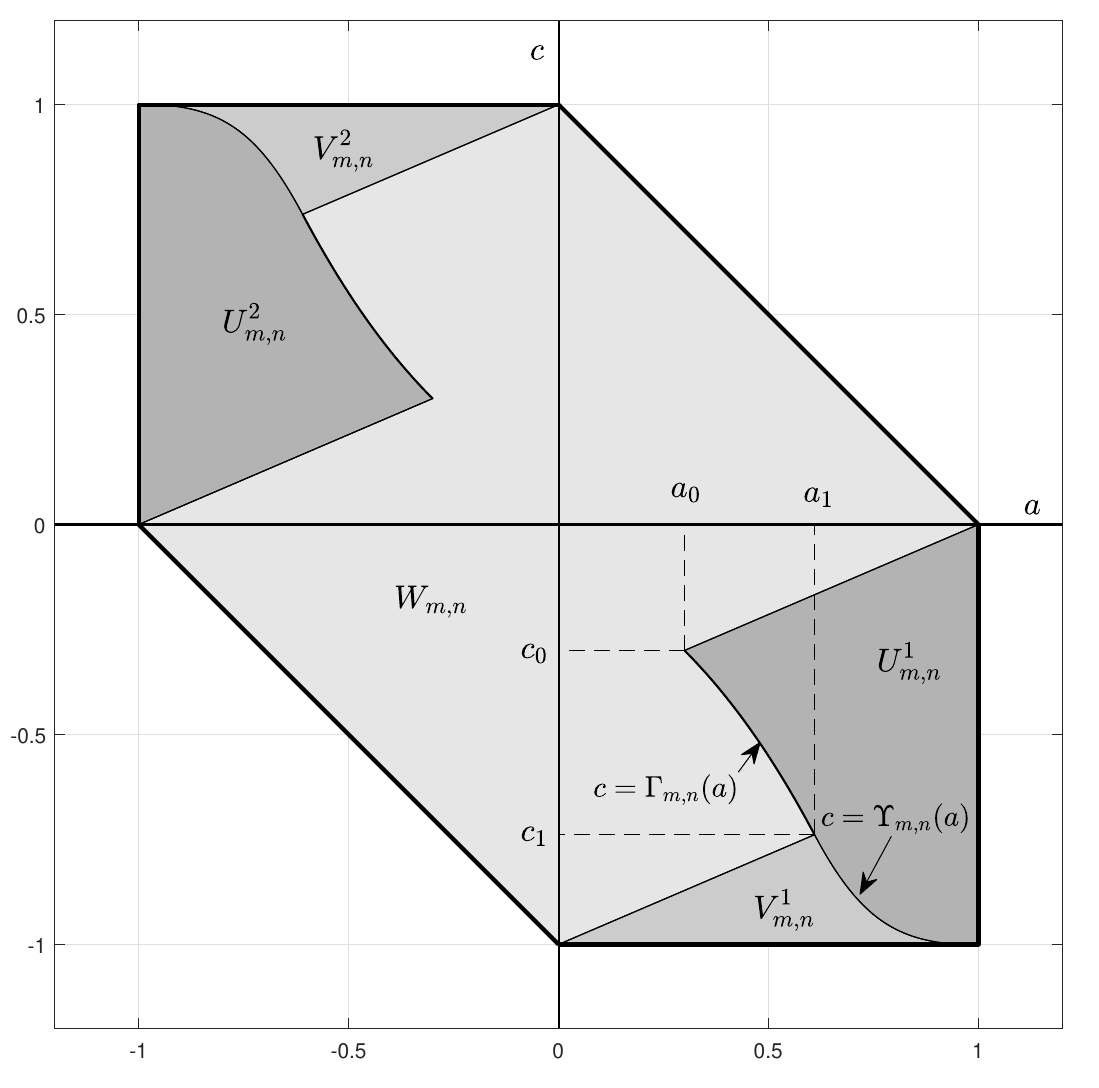}
	\caption{\textcolor{black}{Representation of $\pi_{a,c}({\mathsf S}_{m,n}^h)$.} Here we have considered the case $m=10$ and $n=3$. }\label{fig:Proj_10_3}
\end{figure}

In order to obtain a parametrization of ${\mathsf S}^h_{m,n}$ it will be necessary to know the projection of ${\mathsf S}^h_{m,n}$ over a plane. The most convenient plane is $b=0$ (also called the $ac$ plane) since the unit ball is symmetric with respect to that plane. The latter is justified with the obvious equation $\vertiii{(a,b,c)}_{m,n}=\vertiii{(a,-b,c)}_{m,n}$ whenever $m$ is even and $n$ odd.

In the main result of this section some notations and definitions will be needed. First we consider two curves that will play an important role, namely $c=\Gamma_{m,n}(a)$ and $c=\Upsilon_{m,n}(a)$. The first curve is defined implicitly using the following lemma.

\begin{lemma}\label{lem:lambda_0_Jmn}
	Let $m,n$ be positive integers such that $m\ge 2n$, $m$ is even and $n$ is odd. \textcolor{black}{Define the numbers $J_{m,n}=\frac{m}{n}\left(\frac{n}{m-n}\right)^\frac{m-n}{m}$ and $\lambda_0=\frac{n}{m-n}$} and consider the equation
	\begin{equation}\label{eq:implicit_curve}
	J_{m,n}(1-a)^\frac{m-n}{m}|c|^\frac{n}{m}-1-a-c=0.
	\end{equation}
	Then in \eqref{eq:implicit_curve} $c$ can be expressed as a function of $a$ with $a \in (a_0, a_1)$, namely $\Gamma_{m,n}(a)$, where the curve $c=\Gamma_{m,n}(a)$ meets the parallel lines $c=\lambda_0(a-1)$ and \textcolor{black}{$c=\lambda_0 a -1$} at the points $(a_0,c_0)$ and $(a_1,c_1)$ respectively (see Figure~\ref{fig:Proj_10_3}). Moreover, $\Gamma_{m,n}(a)$ is strictly decreasing and $\Gamma_{m,n}(a)<0$ for $a\in[a_0,a_1]$. It turns out that 
	$$
	a_0=\frac{n}{m}\quad\text{and}\quad c_0=-\frac{n}{m}
	$$
	whereas $a_1$ and $c_1$ cannot generally be obtained explicitly. 
\end{lemma}
\begin{proof}
	One can check easily that $(a_0,c_0)$ with $a_0=\frac{n}{m}$ and $c_0=-\frac{n}{m}$ is a solution to \eqref{eq:implicit_curve} and that the line $c=\lambda_0(a-1)$ passes though $(a_0,c_0)$. 
	Notice that for each fixed $\overline{a} \in (a_0, a_1)$, the function 
	\[
	x \in (-1,0) \mapsto \varphi(x) := J_{m,n} (1 - \overline{a} )^{\frac{m-n}{m}} |x|^{\frac{n}{m}} - 1 - \overline{a} - x 
	\]
	is a strictly decreasing function. It is clear that $\phi (0) = -1 - \overline{a} <0$. Observe also that 
	\[
	z \in (0,1) \mapsto \psi (z):= J_{m,n} (1-z)^{\frac{m-n}{m}} - z
	\] 
	is strictly decreasing; so 
	\begin{align*}
	\phi(-1) = J_{m,n}(1-\overline{a})^{\frac{m-n}{m}} - \overline{a} = \psi (\overline{a}) > \psi (a_1) = \frac{1+ a_1 + c_1 } {|c_1|^{\frac{n}{m}} } - a_1 > 0,
	\end{align*} 
	where the last equality holds since the intersection point $(a_1, c_1)$ satisfies the equation \eqref{eq:implicit_curve}. By Intermediate Value Theorem, this shows that there exists a unique $\overline{c} \in (c_0, c_1)$ such that $\phi(\overline{c})=0$, i.e., $(\overline{a},\overline{c})$ satisfies the equation \eqref{eq:implicit_curve}. 
	
	Next, differentiating \eqref{eq:implicit_curve} with respect to $a$ we obtain
	$$
	c'=G(a,c)=-\frac{\frac{m-n}{m}J_{m,n}\left(\frac{|c|}{1-a}\right)^\frac{n}{m}+1}{J_{m,n}\left(\frac{1-a}{|c|}\right)^\frac{m-n}{m}+1}<0
	$$
	for all $(a,c)\in(a_0,a_1)\times (c_0,c_1)$. Then if $\Gamma_{m,n}$ exists, it must be the solution of the initial value problem for a given point $\overline{a} \in (a_0, a_1)$, 
	\begin{equation}\label{eq:IVP2}
	\begin{cases}
	c'(a)&=G(a,c),\\
	c(\overline{a})&= \overline{c}.
	\end{cases}
	\end{equation}
	Since $\frac{\partial G}{\partial a}$ and $\frac{\partial G}{\partial c}$ are both continuous in $(a_0,a_1)\times (c_0,c_1)$, the Picard-Lindelöf theorem shows that the initial value problem \eqref{eq:IVP2} has a unique solution (which is $C^1$) around some interval containing $\overline{a}$. By uniqueness, this local solution coincides with the one given from Intermediate Value Theorem locally. 
	 Since $\overline{a} \in (a_0,a_1)$ is given arbitrarily, we conclude that $\Gamma_{m,n}(a)$ exists and is well-defined for all $a \in (a_0 , a_1)$. Also $\Gamma_{m,n}$ is strictly decreasing and the curve $c=\Gamma_{m,n}(a)$ meets the straight line $c=\lambda_0a-1$ at the point $(a_1,c_1)$ with $\frac{n}{m}=a_0<a_1<1$ and $-\frac{n}{m}=c_0>c_1>-1$.
\end{proof}
\textcolor{black}{Also, we will consider the curve $c=\Upsilon_{m,n} (a)$ with $a\in[a_1,1]$ where
	$$
	\Upsilon_{m,n} (a)=-\frac{a^\frac{m-n}{n}}{(1-a)^\frac{m-n}{n}+a^\frac{m-n}{n}}=-\frac{1}{1+\left(\frac{1-a}{a}\right)^\frac{m-n}{n}}.
	$$
Interestingly, the curves $c=\Gamma_{m,n}(a)$, $c=\Upsilon_{m,n} (a)$ and the straight line $c=\lambda_0 a-1$ meet at the point $(a_1,c_1)$ (see Figure~\ref{fig:Proj_10_3}).} Also, for every pair $m,n$ of positive integers with $m\ge 2n$, $m$ even and $n$ odd, we will use the sets $\Pi$, $U_{m,n}$, $U^1_{m,n}$, $U^2_{m,n}$, $V_{m,n}$, $V^1_{m,n}$, $V^2_{m,n}$ and $W_{m,n}$ defined as
\begin{align*}
	\Pi&=\{(a,c)\in{\mathbb R}^2:(a,c)\in[-1,1]^2\text{ and }|a+c|\leq 1\},\\
	U_{m,n}&=U^1_{m,n}\cup U^2_{m,n},\\
	V_{m,n}&=V^1_{m,n}\cup V^2_{m,n},\\
\textcolor{black}{	W_{m,n}}&= \textcolor{black}{\Pi \setminus (U_{m,n}\cup V_{m,n})},
\end{align*}
where $U^2_{m,n}=-U^1_{m,n}$, $V^2_{m,n}=-V^1_{m,n}$ and
	\begin{align*}
	U^1_{m,n}&=\big\{(a,c) \in \Pi :a_0\leq a\leq a_1,\ \lambda_0(a-1)\geq c\geq \Gamma_{m,n}(a)\big\}\\
	&\qquad\qquad \cup\big\{(a,c)\in \Pi :a_1\leq a\leq 1,\ \lambda_0(a-1)\geq c\geq \Upsilon_{m,n}(a)\big\},\\
	V^1_{m,n}&=\big\{(a,c) \in \Pi :0\leq a\leq a_1,\ -1\leq c\leq \lambda_0a-1\big\}\\
	&\qquad\qquad \cup\big\{(a,c) \in \Pi :a_1\leq a\leq 1,\ -1\leq c\leq \Upsilon_{m,n}(a)\big\}.
	\end{align*}
See Figure~\ref{fig:Proj_10_3} for a representation of $\Pi$, $U_{m,n}$, $V_{m,n}$ and $W_{m,n}$.

\begin{theorem}\label{thm:projection}
	Let $m,n\in{\mathbb N}$ be such that $m>n$ with $m$ even and $n$ odd. Then
	$$
	\pi_{a,c}({\mathsf S}^h_{m,n})=\Pi.
	$$
\end{theorem}

\begin{proof}
Since ${\mathsf B}_{m,n}^h$ is symmetric with respect to the plane $b=0$, it is clear that $\pi_{a,c}({\mathsf S}^h_{m,n})$ is the intersection of ${\mathsf B}^h_{m,n}$ and the plane $b=0$, that is,
	$$
	\pi_{a,c}({\mathsf S}^h_{m,n})={\mathsf B}^h_{m,n}\cap\{(a,0,c):a,c\in{\mathbb R}\}.
	$$
In other words, $(a,c)\in \pi_{a,c}({\mathsf S}^h_{m,n})$ if and only if $\vertiii{(a,0,c)}_{m,n}\leq 1$. Using \eqref{align:bequalszero} it follows straightforwardly that
	$$
	\pi_{a,c}({\mathsf S}^h_{m,n})=\Pi
	$$
as desired.
\end{proof}

Next we are going to parametrize the unit sphere ${\mathsf S}_{m,n}^h$. Recall that if $m=2n$ then ${\mathcal P}^h_{m,n}({\mathbb R})$ is nothing but the space ${\mathcal P}(^2\ell_\infty^2)$, whose unit ball was already studied by Choi and Kim in \cite{CK1,CK2} providing a formula for the norm of any quadratic form $ax^2+bxy+cy^2$ and a characterization of the extreme polynomials. A parametrization and a representation of the unit sphere of ${\mathcal P}(^2\ell_\infty^2)$ can be found in \cite{JMR2021}. 

Using the parametrization of ${\mathsf S}_{m,n}^h$ mentioned above it will be easy to localize the extreme points of ${\mathsf B}_{m,n}^h$. Let us introduce first some notations needed in this section. For every $m,n\in{\mathbb N}$ with $m>n$, $m$ even, $n$ odd, we define
	 $$
	 J_{m,n}=\frac{m}{n}\left(\frac{n}{m-n}\right)^\frac{m-n}{m}.
	 $$
Also, we shall consider the functions $F_{m,n},G_{m,n}:\Pi\rightarrow {\mathbb R}$ defined as
	\begin{equation}\label{def:Fmn}
	F_{m,n}(a,c)=\begin{cases}
		J_{m,n}\left(1-a\right)^\frac{m-n}{m}|c|^\frac{n}{m}&\text{if $(a,c)\in U^1_{m,n}$,}\\
		J_{m,n}\left(1+a\right)^\frac{m-n}{m}c^\frac{n}{m}&\text{if $(a,c)\in U^2_{m,n}$,}\\	
		J_{m,m-n}\left(1+c\right)^\frac{n}{m}a^\frac{m-n}{m}&\text{if $(a,c)\in V^1_{m,n}$,}\\
		J_{m,m-n}\left(1-c\right)^\frac{n}{m}|a|^\frac{m-n}{m}&\text{if $(a,c)\in V^2_{m,n}$,}\\
		1-|a+c|&\text{if $(a,c)\in W_{m,n}$,}
	\end{cases}
	\end{equation}
whenever $m\ge2n$ and
	$$
	G_{m,n}(a,c)=F_{m,m-n}(c,a)
	$$
if $m\le 2n$. Observe that for all odd $n\in{\mathbb N}$ we have
	$$
	\textcolor{black}{F_{2n,n}(c,a)= G_{2n,n}(a,c)}=
	\begin{cases}
	2\sqrt{c(a-1)}&\text{if $(a,c)\in U^1_{2,1}$,}\\
	2\sqrt{c(1+a)}&\text{if $(a,c)\in U^2_{2,1}$,}\\	
	2\sqrt{a(1+c)}&\text{if $(a,c)\in V^1_{2,1}$,}\\
	2\sqrt{a(c-1)}&\text{if $(a,c)\in V^2_{2,1}$,}\\
	1-|a+c|&\text{if $(a,c)\in W_{2,1}$.}
	\end{cases}
	$$
It was proved in \cite[Theorem 2.4]{JMR2021} that 
	$$
	{\mathsf S}_{2n,n}^h=\graph(F_{2n,n})\cup\graph(-F_{2n,n}).
	$$
More generally, in Theorem \ref{thm:parametrization} we will show that
	$$
	{\mathsf S}_{m,n}^h=
	\begin{cases}
	\graph(F_{m,n})\cup\graph(-F_{m,n})&\text{if $m\ge2n$,}\\
	\graph(G_{m,n})\cup\graph(-G_{m,n})&\text{if $m\le 2n$,}
	\end{cases}
	$$
for every $m,n\in {\mathbb N}$ with $m>2n$, $m$ even and $n$ odd. To prove this the following results will be useful.

\begin{lemma}\label{lem:regions}
	If $m,n\in {\mathbb N}$ are such that \textcolor{black}{$m \geq 2n$}, $m$ is even and $n$ is odd, \textcolor{black}{we define $\Phi:\Pi \cap \{(a,c): a\neq 0 \text{ and } c\neq0\} \rightarrow {\mathbb R}^2$ by} %
	$$
	\Phi(a,c)=\left(\frac{F_{m,n}(a,c)}{a},\frac{n F_{m,n}(a,c)}{mc}\right) 	
	$$
	Then
	\begin{align*}
	&\Phi(V^k_{m,n} \setminus \{ a = 0 \} )={\mathcal A}^k_{m,n} \cup \{(0,0)\},\quad\text{for $k\in\{1,2\}$},\\	
	&\Phi(U^k_{m,n} \setminus \{ c = 0 \} )={\mathcal B}^k_{m,n} \cup \{(0,0)\},\quad\text{for $k\in\{1,2\}$}
		\end{align*}
and
	\[
	\Phi(W_{m,n} \setminus \{ ac =0\} )= [ ( {\mathbb R}^2 \setminus \{bt = 0 \} ) \setminus \left({\mathcal A}_{m,n}\cup{\mathcal B}_{m,n}\right) ] \cup \{ (0,0)\}.	
	\]
\end{lemma}	
	
	\begin{proof}
	 \textcolor{black}{We present the proof for each region.}
	\begin{itemize}
					\item \textcolor{black}{The region $V_{m,n}^k$: By symmetry, it is suffices to show the case $k=1$. Note that 
	\[
	\bigcup_{0 \leq \lambda \leq \lambda_0} \{ (a, \lambda a -1) : 0 < a \leq a_\lambda \} = V_{m,n}^1
	\]
	where $\lambda_0 = \frac{n}{m-n}$ (see Lemma \ref{lem:lambda_0_Jmn}) and $a_\lambda \in (0,1]$ is the value satisfying that $\lambda a_\lambda -1 = \Upsilon_{m,n}(a_\lambda)$ for $0\leq \lambda \leq \lambda_0$. See Figure~\ref{fig:Proj_10_3} or Figure~\ref{fig:Regions_Phi} for a representation of $V_{m,n}^1$ and Figure~\ref{fig:Image_Regions_Phi} for a representation of $\Phi(V_{m,n}^1)$. We shall check that 
	\begin{equation}\label{V_mn_claim}
	\{ \Phi (a, \lambda a - 1) : 0 < a \leq a_\lambda \} = \left\{ \left( b_\lambda , t \right) : \Lambda_{m,n} (b_\lambda) \leq t < 0 \right\},
	\end{equation}
	where 
	\[
	b_\lambda := \frac{m}{m-n} \left( \frac{m-n}{n} \lambda \right)^{\frac{n}{m}}. 
	\]
	 Using \eqref{def:Fmn}, a direct computation shows that for $0 < a \leq a_\lambda$, 
	\[
	\Phi (a, \lambda a - 1) = \left( b_\lambda , \, \frac{n}{m-n} \left( \frac{m-n}{n} \lambda \right)^{\frac{n}{m}} \frac{a}{\lambda a - 1 } \right). 
	\]
	In particular, in the case when $\lambda = 0$, 
	\[
	\Phi (a, -1) = (0,0) \quad \text{ for every } 0 < a \leq 1.
	\]
Note from \eqref{eq:equality1} that $t= \frac{n}{m-n} \left( \frac{m-n}{n} \lambda \right)^{\frac{n}{m}} \frac{a}{\lambda a - 1 } \geq \Lambda_{m,n}(b_\lambda)$ if and only if 
\begin{equation}\label{eq:bt_eq1}
K_{m,n} (b_\lambda)^{\frac{m}{n}} -\frac{n}{m} \frac{b_\lambda}{t} \geq 1 - \left( \frac{m-n}{m} \right) \frac{b_\lambda}{t} \, t^{\frac{m}{m-n}}
\end{equation}
provided that $\lambda \neq 0$. By definition of $b_\lambda$ and $t$, we observe that the left-hand side and right-hand side of \eqref{eq:bt_eq1} is $\frac{1}{a}$ and $1+ \left( \frac{\lambda a}{1-\lambda a} \right)^{\frac{n}{m-n}}$; hence \eqref{eq:bt_eq1} is equivalent to that 
\begin{equation}\label{eq:bt_eq2}
\frac{1}{a} \geq 1 + \left( \frac{\lambda a }{1-\lambda a} \right)^{\frac{n}{m-n}}. 
\end{equation}
Observe that \eqref{eq:bt_eq2} is equivalent to saying that $(a, \lambda a - 1)$ satisfies $\Upsilon_{m,n}(a) \geq \lambda a -1$. Also, the equality in \eqref{eq:bt_eq2} holds if and only if $\Upsilon_{m,n}(a)=\lambda a -1$. This proves the claim \eqref{V_mn_claim}; hence $\Phi(V_{m,n}^1 \setminus \{ a = 0 \}) = \mathcal{A}_{m,n}^1 \cup \{(0,0)\}$.}
\end{itemize}

\begin{itemize} 
					\item \textcolor{black}{The region $U_{m,n}^k$: Again, by symmetry, we only consider the case $k=1$. 
For simplicity, put $\lambda_1:= -\frac{c_1}{1-a_1}$ and consider 
\[
U_{m,n}^1 \setminus \{ c=0\} = U_{m,n,1}^1 \cup U_{m,n,2}^1 \cup \{ (1,c) \in \Pi : -1 \leq c < 0 \},
\]
where
\begin{align*}
&U_{m,n,1}^1 :=  \left \{(a,c) \in U_{m,n} : a_1 \leq a < 1, \, \Upsilon_{m,n}(a) \leq c \leq \lambda_1 (a-1) \right\}, \\
&U_{m,n,2}^1 := \left \{(a,c) \in U_{m,n} : a_0 \leq a < 1, \, \lambda_1 (a-1) < c \leq \lambda_0 (a-1), \, c \geq \Gamma_{m,n}(a) \right\}.
\end{align*} 
See Figure~\ref{fig:Regions_Phi} for a representation of $U_{m,n,1}^1$ and $U_{m,n,2}^1$ and Figure~\ref{fig:Image_Regions_Phi} for a representation of $\Phi(U_{m,n,1}^1)$ and $\Phi(U_{m,n,2}^1)$. 
It is clear that $\Phi(1,c) = (0,0)$ for every $-1 \leq c <0$. 
					\item[$\circ$] The region $U_{m,n,1}^1$: Our claim is to prove 
\begin{equation}\label{Umn11}
\Phi (U_{m,n,1}^1) = \mathcal{B}_{m,n}^1 \cap \{ \tau_0 \leq t < 0 \}.
\end{equation} 
Consider $(a,c) \in U_{m,n,1}^1$ with $c=\lambda (a-1)$. Notice that $\lambda_1$ is the slope of joining $(1,0)$ and $(a_1, c_1)$. Again note from \eqref{def:Fmn} that 
\begin{equation}\label{eq:phi_a_lambda(a-1)}
\Phi (a,\lambda(a-1)) = \left ( \frac{1-a}{a} \frac{m}{n} \left( \frac{n}{m-n} \right)^{\frac{m-n}{m}} \lambda^{\frac{n}{m}} , - \left( \frac{n}{m-n} \right)^{\frac{m-n}{m}} \lambda^{\frac{n-m}{m}} \right). 
\end{equation}
For simplicity, let us put 
\begin{align*}
&b=b(a,\lambda) := \left(\frac{1-a}{a}\right) \frac{m}{n} \left( \frac{n}{m-n} \right)^{\frac{m-n}{m}} \lambda^{\frac{n}{m}};\\
&t=t(\lambda):= - \left( \frac{n}{m-n} \right)^{\frac{m-n}{m}} \lambda^{\frac{n-m}{m}}.
\end{align*}
Notice that $b(a_1, \lambda_1) = \frac{m}{m-n}$. Indeed, we can write 
\begin{align}
b(a_1,\lambda_1) 
= \frac{1-a_1}{a_1} \left( \frac{n}{m-n}\right)^{-\frac{n}{m}} \lambda_1^{\frac{n}{m}} \left( \frac{m}{m-n}\right). \label{eq:ba1lambda1}
\end{align}
On the other hand, since $(a_1, \lambda_1(a_1 -1))$ satisfies \eqref{eq:implicit_curve} and $\frac{n}{m-n}a_1 -1 = c_1$, we have that 
\begin{align*}
\frac{m}{n} \left( \frac{n}{m-n} \right)^{\frac{m-n}{m}} (1-a_1)^{\frac{m-n}{m}} \lambda_1^{\frac{n}{m}} (1-a_1)^{\frac{n}{m}} &= 1 + a_1 + c_1 \\
&= \frac{m}{m-n} a_1,
\end{align*} 
which is equivalent to 
\begin{equation}\label{eq:lambda1}
\left( \frac{n}{m-n}\right)^{-\frac{n}{m}} \left( \frac{1-a_1}{a_1} \right) \lambda_1^{\frac{n}{m}} = 1. 
\end{equation}
Combining this with \eqref{eq:ba1lambda1}, we conclude that $b(a_1,\lambda_1) = \frac{m}{m-n}$. We claim that $t(\lambda_1)$ coincides with $\tau_0$. To this end, observe from \eqref{eq:equality1} that $t(\lambda_1) = \Lambda_{m,n}(b)$ is equivalent to
\begin{align}
\left( \frac{n}{m-n} \right) \left( \frac{m-n}{m}\right)^{\frac{m}{n}} b^{\frac{m}{n}} + \frac{n}{m} & \left( \frac{n}{m-n}\right)^{\frac{n-m}{m}} \lambda_1^{\frac{m-n}{m}} b \nonumber \\
&= 1 + \left(\frac{m-n}{m}\right) b \left( \frac{n}{m-n} \right)^{\frac{n}{m}} \lambda_1^{-\frac{n}{m}}. \label{eq:tau_0_verification}
\end{align} 
Observe that \eqref{eq:tau_0_verification} is satisfied for $b= \frac{m}{m-n}$ if and only if 
\begin{equation}\label{eq:tau_0_verification2}
\frac{n}{m-n} + \left( \frac{n}{m-n}\right)^{\frac{n}{m}} \lambda_1^{\frac{m-n}{m}} = 1 + \left( \frac{n}{m-n}\right) ^{\frac{n}{m}} \lambda_1^{-\frac{n}{m}}. 
\end{equation} 
Taking \eqref{eq:lambda1} into account, we see that \eqref{eq:tau_0_verification2} is equivalent to 
\[
\frac{n}{m-n} + \left( \frac{1-a_1}{a_1}\right) \lambda_1 = 1+ \left( \frac{1-a_1}{a_1} \right) 
\]
which is indeed true because of the fact 
\[
c_1 = \lambda_1 (a_1 - 1) = \frac{n}{m-n} a_1 -1. 
\]
Consequently, we proved that \eqref{eq:tau_0_verification} is satisfied when $b= \frac{m}{m-n}$, which implies that $t(\lambda_1) = \Lambda_{m,n} (\frac{m}{m-n}) = \tau_0$. 
Summarizing, we have proved that
\begin{equation}\label{eq:lambda=lambda1}
\Phi (a, \lambda_1(a-1)) = \left ( b(a, \lambda_1) , \, \tau_0 \right) \quad \text{ for all } a_1 \leq a < 1 
\end{equation}
and $0 < b(a,\lambda_1) \leq b(a_1,\lambda_1) = \frac{m}{m-n}$ for $a_1 \leq a < 1$. 
Next, pick 
	$$
	(a,\lambda(a-1)) \in U_{m,n,1}^1
	$$ 
with $\lambda > \lambda_1$. 
It is enough to check that the corresponding $(b(a,\lambda), t(\lambda))$ satisfies that $t(\lambda) \leq \Lambda_{m,n} (b(a,\lambda))$ which is equivalent to
\begin{equation}\label{eq:Bmn_verification}
K_{m,n} b(a,\lambda)^{\frac{m}{n}} - \frac{n}{m} \frac{b(a,\lambda)}{t(\lambda)} - 1 - \left( \frac{m-n}{m} \right) b(a,\lambda) t(\lambda)^{\frac{n}{m-n}} \leq 0.
\end{equation}
By definition of $b(a,\lambda)$ and $t(\lambda)$, \eqref{eq:Bmn_verification} is equivalent to 
\[
\left( \frac{1-a}{a} \right)^{\frac{m}{n}} \lambda + \left( \frac{1-a}{a}\right) \lambda - 1 - \left( \frac{1-a}{a}\right) \leq 0
\]
which is indeed true because $\lambda(a-1) \geq \Upsilon_{m,n}(a)$. Also, the equality in \eqref{eq:Bmn_verification} holds if and only if $\lambda(a-1)=\Upsilon_{m,n}(a)$. It follows that \eqref{Umn11} holds and the claim is proved.
						\item[$\circ$] The region $U_{m,n,2}^1$: We claim that 
\begin{equation}\label{Umn12}
\left\{ \Phi (a,c): (a,c) \in U_{m,n,2}^1 \right\} = \mathcal{B}_{m,n}^1 \cap \{ -1 \leq t < \tau_0 \}.
\end{equation}
Note first that $a_0 = \frac{n}{m}$ and $c_0 = - \frac{n}{m}$; hence $-\frac{c_0}{1-a_0} = \frac{n}{m-n}$. 
Consider $(a, c) \in U_{m,n,2}^1$ with $c=\lambda(a-1)$. Recall from \eqref{eq:phi_a_lambda(a-1)} that if $\lambda = \frac{n}{m-n}$, then 
\[
0< b(a,\lambda) = \frac{1-a}{a} \frac{m}{m-n} \leq b\left( a_0 , \lambda \right) = \frac{m-n}{n} \frac{m}{m-n} = \frac{m}{n} 
\]
and $t(\lambda)=-1$. Now, let $\frac{n}{m-n} < \lambda < \lambda_1$. We claim that the corresponding $b(a,\lambda)$ and $t(\lambda)$ satisfies that 
\begin{equation}\label{eq:balambda}
b(a,\lambda) \leq 2m \frac{t(\lambda)}{(m-n)|t(\lambda)|^{\frac{m}{m-n}} - m |t(\lambda)| - n } = g(t(\lambda)).
\end{equation} 
Note that \eqref{eq:balambda} is equivalent to
\begin{equation}\label{eq:balambda2}
\frac{ -(m-n)|t(\lambda)|^{\frac{m}{m-n}} + m |t(\lambda)| + n}{m} \leq - \frac{2t(\lambda)}{b(a,\lambda)}. 
\end{equation} 
As a matter of fact, \eqref{eq:balambda} indeed holds since $c \geq \Gamma_{m,n}(a)$, that is, 
\[
\frac{m}{n} \left( \frac{n}{m-n}\right)^{\frac{m-n}{m}} (1-a) \lambda^{\frac{n}{m}} \leq 1+a+\lambda(a-1).
\]
Thus, 
\begin{align*}
&\frac{ -(m-n)|t(\lambda)|^{\frac{m}{m-n}} + m |t(\lambda)| + n}{m} \\
&\hspace{4em} = - \frac{1}{\lambda} \frac{n}{m} + \left( \frac{n}{m-n}\right)^{\frac{m-n}{m}} \lambda^{\frac{n-m}{m}} + \frac{n}{m} \\
&\hspace{4em} \leq - \frac{n}{\lambda m } + \frac{n}{\lambda m} \frac{1}{1-a} (1+a+\lambda(a-1)) +\frac{n}{m} \\ 
&\hspace{4em}= \frac{n}{\lambda m } \left( -1 + \frac{1}{1-a} + \frac{a}{1-a} \right) \\
&\hspace{4em}= \frac{2a n }{\lambda(1-a)m} = -\frac{2t(\lambda)}{b(a,\lambda)}. 
\end{align*} 
Combining \eqref{Umn11}, \eqref{Umn12} with the fact $\Phi(1,c)=(0,0)$ for every $-1\leq c <0$, we conclude $\Phi(U_{m,n}^1 \setminus \{ c = 0 \} )=\mathcal{B}_{m,n}^1 \cup \{ (0,0) \}$. 
	\item The region $W_{m,n}$: By symmetry it will be enough to consider the region $W_{m,n} \cap \{(a,c)\in\Pi : a >0 \}$. 
						\item[$\circ$] The region $P_1:= W_{m,n} \cap \{(a,c)\in\Pi: a,c > 0\}$ (see Figure~\ref{fig:Regions_Phi}): 
	Consider $(a,c) \in P_1$ with $a+c=k$ and $0<k\leq 1$. Then 
\[
\Phi (a,c) = \left( \frac{1-k}{a}, \frac{n}{m} \frac{1-k}{k-a} \right). 
\] 
Then for each $0 < k \leq 1$
\[
\{ \Phi (a,c) : a,c > 0 \text{ and } a+c=k\} = \left\{ \left( \frac{1-k}{a}, \frac{n}{m} \frac{1-k}{k-a} \right) : 0 < a < k \right\}.
\]
It is an easy exercise to prove that 
\[
\bigcup_{0 < k < 1} \left\{ \left( \frac{1-k}{a}, \frac{n}{m} \frac{1-k}{k-a} \right) : 0 < a < k \right\} = \{ (b,t)\in\mathbb{R}^2 : b,t > 0\}
\]
which implies that 
\[
\{ \Phi (a,c) : (a,c) \in P_1 \} = \{ (b,t)\in\mathbb{R}^2 : b,t > 0\} \cup \{(0,0)\}. 
\]
			\item[$\circ$] The region $P_2:= W_{m,n} \cap \{(a,c)\in\Pi: a>0, \, c < 0\}$: We divide $P_2$ into three regions: $P_2 = P_{2,1} \cup P_{2,2} \cup P_{2,3}$, where 
		\begin{align*}
		&P_{2,1} = \left\{ (a,c) \in \Pi: 0 < a<1, \, c = \lambda a - 1, \, \frac{n}{m-n} < \lambda \leq \frac{m-n}{n}, \, c < \Gamma_{m,n}(a) \right\}; \\
		&P_{2,2} = \left\{ (a,c) \in \Pi : 0 < a<1, \, c = \lambda a - 1, \, \lambda \geq \frac{m-n}{n}, \, a+c \leq 0 \right\}; \\
		&P_{2,3} = \left\{ (a,c) \in \Pi : 0<a<1, \, c = \lambda(a - 1), 0 < \lambda < \frac{n}{m-n}, \, a+c \geq 0 \right\}. 
		\end{align*}
	See Figure~\ref{fig:Regions_Phi} for a representation of $P_{2,1}$, $P_{2,2}$ and $P_{2,3}$ and Figure~\ref{fig:Image_Regions_Phi} for a representation of $\Phi(P_{2,1})$, $\Phi(P_{2,2})$ and $\Phi(P_{2,3})$.
					\item[--] The region $P_{2,1}$: Pick a point $(a,c) \in P_{2,1}$ with $c=\lambda a -1$. Observe that 
\[
\Phi (a,\lambda a - 1) = \left( 1+ \lambda , \frac{n}{m} \frac{a(1+\lambda)}{\lambda a - 1} \right). 
\]
It is clear that $\frac{m}{m-n} < 1+\lambda \leq \frac{m}{n}$. Observe that $(b,t) = \left( 1+ \lambda , \frac{n}{m} \frac{a(1+\lambda)}{\lambda a - 1} \right)$ satisfies that $b > g(t)$ if and only if 
\begin{equation}\label{eq:b_geq_g(t)}
(\dagger) := (m-n) \left ( \frac{n}{m} \frac{a(1+\lambda)}{|\lambda a - 1 | } \right)^{\frac{m}{m-n}} - n \frac{a(1+\lambda)}{|\lambda a - 1 |} - n < \frac{2an}{\lambda a - 1 }. 
\end{equation}
Indeed, \eqref{eq:b_geq_g(t)} holds since $c < \Gamma_{m,n}(a)$, that is, 
\begin{equation}\label{eq:ac_condition_f}
\frac{m}{n} \left( \frac{n}{m-n} \right)^{\frac{m-n}{m}} (1-a)^{\frac{m-n}{m}} |\lambda a - 1|^{\frac{n}{m}} > a (1+\lambda) 
\end{equation}
More precisely, applying \eqref{eq:ac_condition_f} to $(\dagger)$, we have 
\begin{align*}
(\dagger) < \frac{n}{|\lambda a - 1|} (1-a)- n \frac{a(1+\lambda)}{|\lambda a - 1 |} - n = \frac{2an}{\lambda a -1}.
\end{align*} 
Moreover, if $\lambda = \frac{m-n}{n}$, then
\[
\left\{ \Phi \left( a, \frac{m-n}{n} a -1 \right) : 0 < a \leq \frac{n}{m}\right\} = \left\{ \left(\frac{m}{n}, t\right) : -1 \leq t < 0\right\};
\]
hence 
\begin{equation}\label{claim_nontrivial_W}
\Phi (P_{2,1})= \left\{ (b,t) \in \mathbb{R}^2 : \frac{m}{m-n} < b \leq \frac{m}{n}, \, t < 0, \, b > g(t) \right\}.
\end{equation} 
						\item[--] The region $P_{2,2}$: Pick a point $(a,c) \in P_{2,2}$ with $c=\lambda a -1$.  
Recall that $\Phi (a,\lambda a - 1) = \left( 1+ \lambda , \frac{n}{m} \frac{a(1+\lambda)}{\lambda a - 1} \right) =: (b,t)$. Note that $b = 1+\lambda \geq \frac{m}{n}$ and $a+c=a(1+\lambda)\leq 1$; hence for $0<a <1$ we have 
\[
0 > t = \frac{n}{m} \frac{a(1+\lambda)}{\lambda a -1} \geq \frac{n}{m} \left( -\frac{1}{a}\right) a (1+\lambda) = - \frac{n}{m} b.
\]
This proves that 
\begin{equation}\label{claim_nontrivial_W2}
\Phi (P_{2,2}) = \left\{ (b,t) : b \geq \frac{m}{n}, \, -\frac{n}{m}b \leq t < 0 \right\}.
\end{equation}
					\item[--] The region $P_{2,3}$: Let $(a,c)\in P_{2,3}$ with $c=\lambda (a-1)$ and notice that 
\[
\Phi (a, \lambda(a-1)) = \left( \frac{(1+\lambda)(1-a)}{a}, -\frac{n}{m} \left(\frac{1+\lambda}{\lambda} \right) \right) =: (b,t) \]
and $\lambda \neq 0$. 
It is clear that $t= -\frac{n}{m}\left(\frac{1+\lambda}{\lambda}\right) < -1$ since $\lambda < \frac{n}{m-n}$. 
Moreover, as $a+c=a+\lambda(a-1)\geq 0$, we have 
\[
b= (1+\lambda)\left( \frac{1+\lambda}{\lambda}\right) \leq \frac{1+\lambda}{\lambda} = -\frac{m}{n} t,
\]
that is, $t \leq -\frac{n}{m}b$. This verifies the equality 
\begin{equation}\label{claim_nontrivial_W3}
\Phi (P_{2,3}) = \left\{ (b,t) \in \mathbb{R}^2 : b>0, \, t < -1, \, t \leq -\frac{n}{m}b \right\}.
\end{equation} 
Consequently, by symmetry again, we complete the proof of the assertion 
\[
\Phi(W_{m,n} \setminus \{ ac =0\} )= [ ( {\mathbb R}^2 \setminus \{bt = 0 \} ) \setminus \left({\mathcal A}_{m,n}\cup{\mathcal B}_{m,n}\right) ] \cup \{ (0,0)\}.	
\] \qedhere
}
\end{itemize}
	\end{proof} 

\begin{figure}
	\centering
	\includegraphics[height=.8\textwidth,keepaspectratio=true]{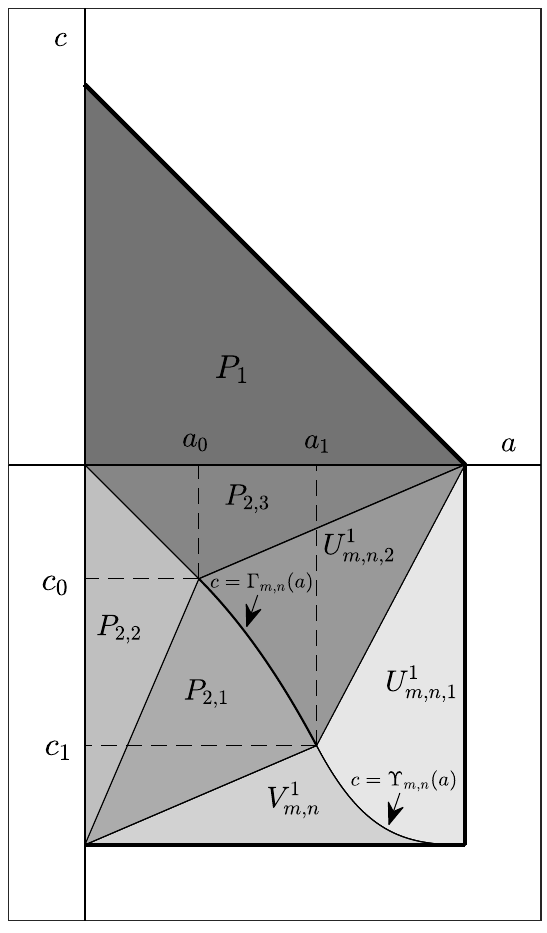}
	\caption{Representation of the regions $P_1$, $P_2$, $P_{2,1}$, $P_{2,2}$, $P_{2,3}$, $U^1_{m,n,1}$, $U^1_{m,n,2}$ and $V_{m,n}^1$ employed in Lemma~\ref{lem:regions} and its proof. Here we have considered the case $m=10$ and $n=3$. }\label{fig:Regions_Phi}
\end{figure}

\begin{figure}
	\centering
	\includegraphics[width=\textwidth,keepaspectratio=true]{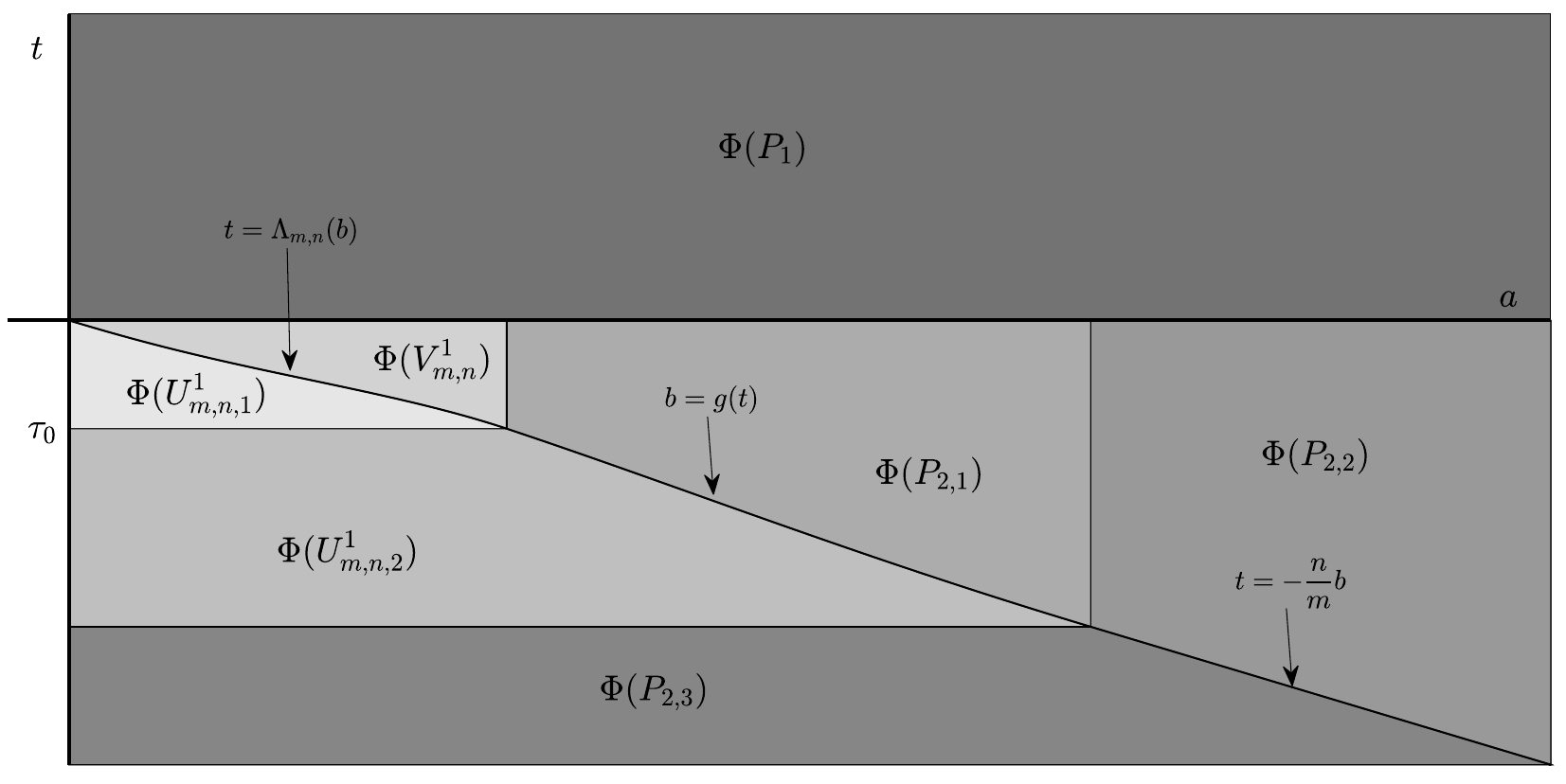}
	\caption{Representation of the image by $\Phi$ of the regions $P_1$, $P_2$, $P_{2,1}$, $P_{2,2}$, $P_{2,3}$, $U^1_{m,n,1}$, $U^1_{m,n,2}$ and $V_{m,n}^1$ employed in Lemma~\ref{lem:regions} and its proof. Here we have considered the case $m=10$ and $n=3$. }\label{fig:Image_Regions_Phi}
\end{figure}
	
\begin{theorem}\label{thm:parametrization}
	For every $m,n\in {\mathbb N}$ with $m>n$, $m$ even, $n$ odd we have
	$$
	{\mathsf S}_{m,n}^h=
	\begin{cases}
	\graph(F_{m,n})\cup\graph(-F_{m,n})&\text{if $m\ge2n$,}\\
	\graph(G_{m,n})\cup\graph(-G_{m,n})&\text{if $m\le 2n$.}
	\end{cases}
	$$
	The reader can find a representation of ${\mathsf S}_{m,n}^h$ in Figure~\ref{fig:Ball_10_3} for the case $m=10$ and $n=3$.
\end{theorem}

\begin{proof}
According to \eqref{eq:reduction} we have that 
$$
\vertiii{(a,b,c)}_{m,n}=\vertiii{(c,b,a)}_{m,m-n},
$$
which allows us to focus our efforts on the case where $m\ge 2n$. As a matter of fact we can assume that $m> 2n$.

Observe that any two norms $\|\cdot\|_a$ and $\|\cdot\|_b$ coincide on a linear space if and only if $\|x\|_a=1$ for every $x$ such that $\|x\|_b=1$. Since the mappings $F_{m,n}$ can be easily proved to be concave using elementary differential calculus, the centrally symmetric set
	$$
	{\mathsf S}_{m,n}^{h,*}=\graph(F_{m,n})\cup\graph(-F_{m,n})
	$$
must be the unit sphere of a norm in ${\mathbb R}^3$. Hence, in order to prove that ${\mathsf S}^h_{m,n}={\mathsf S}^{h,*}_{m,n}$ we just need to show that ${\mathsf S}^{h,*}_{m,n}\subset {\mathsf S}^h_{m,n}$. Actually it suffices to show that $\graph(F_{m,n})\subset {\mathsf S}^h_{m,n}$ due to the symmetry of ${\mathsf S}^h_{m,n}$ with respect to the plane $b=0$. Let us take $(a,b,c)\in \graph(F_{m,n})$, that is 
	$$
	b=F_{m,n}(a,c).
	$$
We will study separately the five cases $(a,c)\in U_{m,n}^k$, $(a,c)\in V_{m,n}^j$, $1\leq k,j\leq 2$ and $(a,c)\in W_{m,n}$.

\noindent\fbox{Case 1: $(a,c)\in U_{m,n}^1$}

From Lemma \ref{lem:regions} we have that
	$$
	\Phi(a,c)=\left(\frac{F_{m,n}(a,c)}{a},\frac{n F_{m,n}(a,c)}{mc}\right)\in{\mathcal B}^1_{m,n}.
	$$
Then according to \eqref{eq:norm} and the fact that
	$$
	F_{m,n}(a,c)=J_{m,n}(1-a)^\frac{m-n}{m}|c|^\frac{n}{m},
	$$
it follows that
	\begin{align*}
	\vertiii{(a,\pm b,c)}_{m,n}&=\left|K_{m,m-n}c\left|\frac{F_{m,n}(a,c)}{c}\right|^\frac{m}{m-n}-a\right|\\
	&=\left|K_{m,m-n}\left(J_{m,n}\right)^\frac{m}{m-n}\frac{c}{|c|}(1-a)-a\right|\\
	&=\left|\frac{c}{|c|}(1-a)-a\right|=|a-1-a|=1.
	\end{align*}
	In the last two steps we took into consideration that $K_{m,m-n}\left(J_{m,n}\right)^\frac{m}{m-n}=1$ and that $c<0$ whenever $(a,c)\in U_{m,n}^1$.
	
\noindent\fbox{Case 2: $(a,c)\in U_{m,n}^2$}

By symmetry, $(-a,-c)\in U^1_{m,n}$ and $F_{m,n}(-a,-c)=F_{m,n}(a,c)$. Then, using the previous case
	$$
	\vertiii{(a,b,c)}_{m,n}=\vertiii{(-a,-b,-c)}_{m,n}=\vertiii{(-a,b,-c)}_{m,n}=1.
	$$
	
\noindent\fbox{Case 3: $(a,c)\in V_{m,n}^1$}

From Lemma \ref{lem:regions} we have that
	$$
	\Phi(a,c)=\left(\frac{F_{m,n}(a,c)}{a},\frac{n F_{m,n}(a,c)}{mc}\right)\in{\mathcal A}^1_{m,n}.
	$$
Then according to \eqref{eq:norm} and the fact that
	$$
	F_{m,n}(a,c)=J_{m,m-n}\left(1+c\right)^\frac{n}{m}a^\frac{m-n}{m},
	$$
it follows that
	\begin{align*}
	\vertiii{(a,\pm b,c)}_{m,n}&=\left|K_{m,n}a\left|\frac{F_{m,n}(a,c)}{a}\right|^\frac{m}{n}-c\right|\\	
	&=\left|K_{m,n}\left(J_{m,m-n}\right)^\frac{m}{n}(1+c)-c\right|=1.
	\end{align*}
In the last two steps we took into consideration that $K_{m,n}\left(J_{m,m-n}\right)^\frac{m}{n}=1$ and that $a\ge 0$ whenever $(a,c)\in V_{m,n}^1$.

\noindent\fbox{Case 4: $(a,c)\in V_{m,n}^2$}

By symmetry, $(-a,-c)\in V^1_{m,n}$ and $F_{m,n}(-a,-c)=F_{m,n}(a,c)$. Then, using the previous case
	$$
	\vertiii{(a,b,c)}_{m,n}=\vertiii{(-a,-b,-c)}_{m,n}=\vertiii{(-a,b,-c)}_{m,n}=1.
	$$
	
\noindent\fbox{Case 5: $(a,c)\in W_{m,n}$}

From Lemma \ref{lem:regions} we have that
	$$
	\Phi(a,c)=\left(\frac{F_{m,n}(a,c)}{a},\frac{n F_{m,n}(a,c)}{mc}\right)\notin{\mathcal A}_{m,n}\cup{\mathcal B}_{m,n}.
	$$
Assume that $b=F_{m,n}(a,c)=1-|a+c|=0$ and $ac\leq 0$. Observe that $a$ and $c$ have opposite (or null) signs and that $|a+c|\leq 1$. Moreover $|a+c|=1$ only if $a=\pm 1$ and $c=0$ or $a=0$ and $c=\pm 1$. The using \eqref{eq:norm} we easily have
	$$
	\vertiii{(\pm 1,0,0)}_{m,n}=\vertiii{(0,0,\pm 1)}_{m,n}=1.
	$$
Finally, if $b=F_{m,n}(a,c)=1-|a+c|\ne 0$ or $ac>0$ we have in both cases that $b=1-|a+c|\geq 0$. Then, according to \eqref{eq:norm} 
it follows that
	\begin{align*}
	\vertiii{(a, b,c)}_{m,n}&=|a+c|+|b|=|a+c|+1-|a+c|=1.
	\end{align*}
\end{proof}

\begin{figure}
	\centering
	\includegraphics[height=.8\textwidth,keepaspectratio=true]{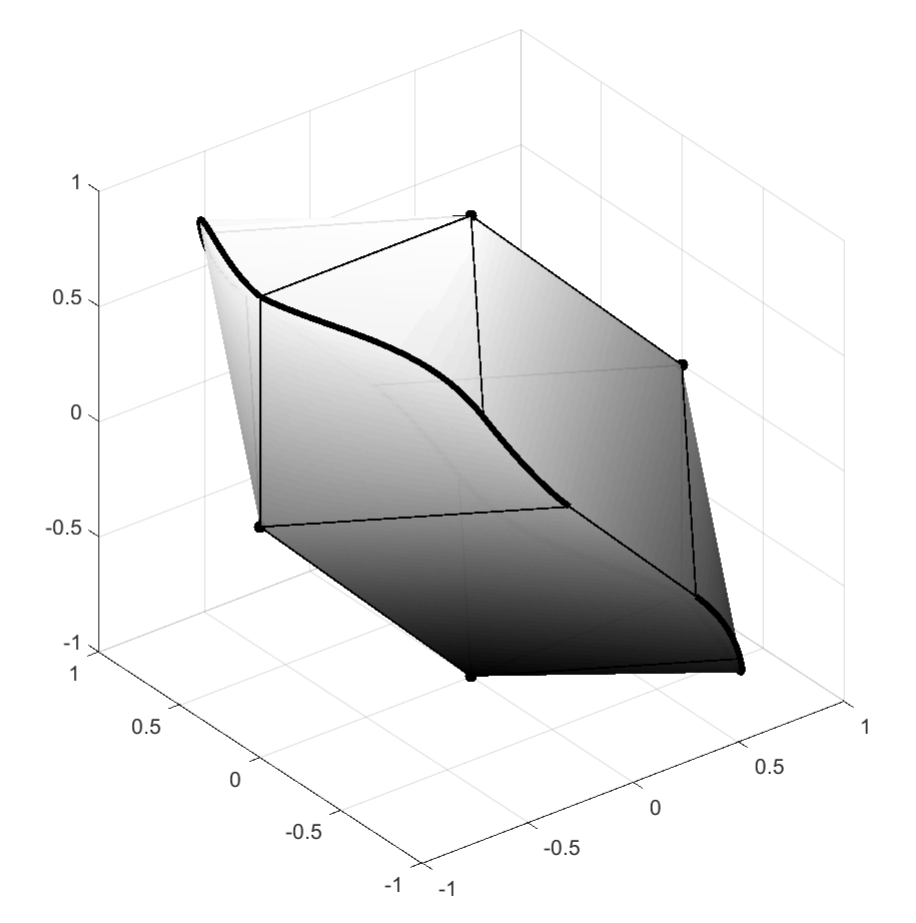}
	\caption{Representation of ${\mathsf S}_{m,n}^h$. Here we have considered the case $m=10$ and $n=3$. }\label{fig:Ball_10_3}
\end{figure}

\section{The extreme points of ${\mathsf B}_{m,n}^h$}\label{sec:extreme}

The parametrization of ${\mathsf S}_{m,n}^h$ found in Theorem~\ref{thm:parametrization} together with the fact that most of the surfaces that take part in the graph of $F_{m,n}$ are ruled surfaces help tremendously to localize the extreme points of ${\mathsf B}_{m,n}^h$. This is addressed in the main result of this section.

\begin{theorem}\label{thm_extreme_points}
	Let $m,n$ be positive integers with $m\ge 2n$, $m$ even and $n$ odd. Then $\ext({\mathsf B}_{m,n}^h)$ consists of the points
	$$
	\pm (1,0,0), \quad \pm (0,0,1),
	$$
	$$
	\pm \left(a,\pm J_{m,n}\frac{(1-a)^\frac{m-n}{m}a^\frac{m-n}{m}}{(1-a)^\frac{m-n}{n}+a^\frac{m-n}{n}},-\frac{a^\frac{m-n}{n}}{(1-a)^\frac{m-n}{n}+a^\frac{m-n}{n}}\right),\quad \text{with $a\in\left[a_1,1\right]$}
	$$
and
	$$
	\pm\left(a,1-|a+\Gamma_{m,n}(a)|,\Gamma_{m,n}(a)\right),\quad \text{with $a\in\left[\frac{n}{m},a_1\right]$},
	$$
	where $J_{m,n}=\frac{m}{n}\left(\frac{n}{m-n}\right)^\frac{m-n}{m}$, $a_1$ is the positive number and $c=\Gamma_{m,n}(a)$ is the curve from Lemma \ref{lem:lambda_0_Jmn}.
\end{theorem}

\begin{proof}
First we show that the graph of $\pm F_{m,n}$ restricted to $V_{m,n}$ is a ruled surface (to recall, the function $F_{m,n}$ is defined in \eqref{def:Fmn}, and the region $V_{m,n}$ is illustrated in Figure \ref{fig:Proj_10_3}). Due to the symmetry of $F_{m,n}$ and $V_{m,n}$, we just need to prove this by restricting $\pm F_{m,n}$ to $V_{m,n}^1$. Consider the lines in the plane $ac$ that pass though $(0,-1)$ and have slope $\lambda$ with $\lambda\in[0,\lambda_0]$, where $\lambda_0 = \frac{n}{m-n}$ as in Lemma \ref{lem:lambda_0_Jmn}. The equation of each of these lines is given by $c=\lambda a-1$. Let $(\eta(\lambda),\xi(\lambda))$ be the intersection point of $c=\lambda a-1$ and the curve $c=\Upsilon_{m,n}(a)$. Obviously, the segments $L_\lambda$ that join $(0,-1)$ and $(\eta(\lambda),\xi(\lambda))$, namely 
	$$
	L_\lambda=\{(a,\lambda a-1):a\in[0,\eta(\lambda)]\}, 
	$$
cover the whole set $V_{m,n}^1$ as $\lambda$ ranges over the interval $[0,\lambda_0]$. Also, $\pm F_{m,n}$ restricted to $L_\lambda$ is linear since, by \eqref{V_mn_claim}
	$$
	\pm F_{m,n}(a,\lambda a-1)=\pm\frac{m}{m-n} \left(\frac{m-n}{n}\lambda\right)^\frac{n}{m}a
	$$
for $a\in[0,\eta(\lambda)]$. 

A similar argument reveals that the graph of $\pm F_{m,n}$ restricted to $U_{m,n}$ is a ruled surface as well. In this case we have to consider the straight lines passing through $(1,0)$ in the $ac$ plane with slope $\lambda\in[\lambda_0,\infty)$. The equation of each of those lines is $c=\lambda(a-1)$. Let $(\eta(\lambda),\xi(\lambda))$ be the intersection point of $c=\lambda(a-1)$ and the curve $c=\Upsilon_{m,n}(a)$ or $c=\Gamma_{m,n}(a)$. Again, the segments $L_\lambda$ that join $(1,0)$ and $(\eta(\lambda),\xi(\lambda))$ cover the whole $U_{m,n}^1$ as $\lambda\ge \lambda_0$. Now it is easy to see that the restriction of $\pm F_{m,n}$ to $L_\lambda$ is affine since, by \eqref{eq:phi_a_lambda(a-1)}
	$$
	\pm F_{m,n}(a,\lambda( a-1))=(1-a) \frac{m}{n} \left( \frac{n}{m-n} \right)^{\frac{m-n}{m}} \lambda^{\frac{n}{m}} .
	$$
Finally, the graph of $\pm F_{m,n}$ when restricted to $W_{m,n}$ is formed by two flat surfaces. We conclude that all the extreme points of ${\mathsf B}^h_{m,n}$ must lie in the intersections of the graphs of $\pm F_{m,n}$ when restricted to $V^k_{m,n}$, $U^k_{m,n}$ with $k=1,2$ and $W_{m,n}$. Some of these intersections produce, in their turn, segments. That is the case of the following intersections:
\begin{itemize}
	\item $\graph(F_{m,n}|_{U^1_{m,n}})\cap \graph(-F_{m,n}|_{U^1_{m,n}})=[P_1,P_3]$,
	
	\item $\graph(F_{m,n}|_{U^2_{m,n}})\cap \graph(-F_{m,n}|_{U^2_{m,n}})=[-P_1,-P_3]$,
	
	\item $\graph(F_{m,n}|_{V^1_{m,n}})\cap \graph(-F_{m,n}|_{V^1_{m,n}})=[P_2,P_3]$,
	
	\item $\graph(F_{m,n}|_{V^2_{m,n}})\cap \graph(-F_{m,n}|_{V^2_{m,n}})=[-P_2,-P_3]$,
	
	\item $\graph(F_{m,n}|_{W_{m,n}})\cap \graph(-F_{m,n}|_{W_{m,n}})=[P_2,-P_1]\cup [-P_2,P_1]$,
	
	\item $\graph(F_{m,n}|_{V^1_{m,n}})\cap \graph(F_{m,n}|_{W_{m,n}})=[P_2,Q_1]$,
	
	\item $\graph(-F_{m,n}|_{V^1_{m,n}})\cap \graph(-F_{m,n}|_{W_{m,n}})=[P_2,Q_2]$,
	
	\item $\graph(F_{m,n}|_{V^2_{m,n}})\cap \graph(F_{m,n}|_{W_{m,n}})=[-P_2,-Q_1]$,
	
	\item $\graph(-F_{m,n}|_{V^2_{m,n}})\cap \graph(-F_{m,n}|_{W_{m,n}})=[-P_2,-Q_2]$,
\end{itemize}
where $P_1=(1,0,0)$, $P_2=(0,0,-1)$, $P_3=(1,0,-1)$, $Q_1=(a_0,1,c_0)$ and $Q_2=(a_0,-1,c_0)$. Here, $(a_0, c_0)$ is the point appeared in Lemma \ref{lem:lambda_0_Jmn}. Those segments can only contain extreme points of ${\mathsf B}^h_{m,n}$ at their endpoints, that is, at $\pm P_1$, $\pm P_2$, $\pm P_3$, $\pm Q_1$,$\pm Q_2$.

There are other intersections that provide both segments and curved lines. That happens in the following cases:
\begin{itemize}
	\item $\graph(\pm F_{m,n}|_{U^k_{m,n}})\cap \graph(\pm F_{m,n}|_{W_{m,n}})$ with $k=1,2$. These intersections include the curved lines
	\begin{equation}\label{eq:Gamma_ext_points}
	\pm\left(a,\pm F_{m,n}(a,\Gamma_{m,n}(a)),\Gamma_{m,n}(a)\right)\quad\text{with $a\in[a_0,a_1]$}
	\end{equation}
	together with the four segments
	$$
	\pm[P_1,Q_1] \quad\text{and}\quad \pm[P_1,Q_2],
	$$
Observe that $\pm Q_1$ and $\pm Q_2$ are already in the curves \eqref{eq:Gamma_ext_points}. To see this we just need to put $a=a_0$ in \eqref{eq:Gamma_ext_points}. Using the definition of $F_{m,n}$, \eqref{eq:Gamma_ext_points} reduces to
	$$
	\pm\left(a,\pm (1-|a+\Gamma_{m,n}(a)|),\Gamma_{m,n}(a)\right)\quad\text{with $a\in\left[\frac{n}{m},a_1\right]$}
	$$
\end{itemize}
Finally there are intersections that provide only curved lines:
\begin{itemize}
	\item $\graph(\pm F_{m,n}|_{U^k_{m,n}})\cap \graph(\pm F_{m,n}|_{V_{m,n}})$ with $k=1,2$.
	
	These intersections provide the curves
	\begin{equation}\label{eq:Upsilon_ext_points}
	\pm\left(a,\pm F_{m,n}(a,\Upsilon_{m,n}(a)),\Upsilon_{m,n}(a)\right)\quad\text{with $a\in[a_1,1]$}
	\end{equation}
or equivalently
	$$
	\pm \left(a,\pm J_{m,n}\frac{(1-a)^\frac{m-n}{m}a^\frac{m-n}{m}}{(1-a)^\frac{m-n}{n}+a^\frac{m-n}{n}},-\frac{a^\frac{m-n}{n}}{(1-a)^\frac{m-n}{n}+a^\frac{m-n}{n}}\right),\quad \text{with $a\in\left[a_1,1\right]$}.
	$$
Observe that putting $a=1$ in \eqref{eq:Upsilon_ext_points} we obtain $\pm P_3$ so the points $\pm P_3$ are already included in the curves \eqref{eq:Upsilon_ext_points}.
\end{itemize}

Therefore the extreme points of ${\mathsf B}_{m,n}^h$ must lie among the points $\pm P_1$ and $\pm P_2$ or in the curves described in \eqref{eq:Gamma_ext_points} and \eqref{eq:Upsilon_ext_points}.

In the rest of the proof, we will say that a plane $H$ in ${\mathbb R}^3$ is a supporting hyperplane to ${\mathsf B}_{m,n}^h$ at $P\in {\mathsf B}_{m,n}^h$ if $H\cap {\mathsf B}_{m,n}^h=\{P\}$. 
To finish the proof we just have to construct supporting planes to ${\mathsf B}_{m,n}^h$ at each point of the 8 curves represented in \eqref{eq:Gamma_ext_points} and \eqref{eq:Upsilon_ext_points} and the four points $\pm P_1$ and $\pm P_2$. The existence of such a plane would guarantee that $P$ is an extreme point of ${\mathsf B}_{m,n}^h$.

Starting with the points $\pm P_1$ and $\pm P_2$, it can be easily seen that the planes
\begin{align*}
2(a-1)+c&=0,\\
2(a+1)+c&=0,\\
a+2(c+1)&=0,\\
a+2(c-1)&=0,
\end{align*}
are supporting planes to ${\mathsf B}_{m,n}^h$ at $P_1$, $-P_1$, $P_2$ and $-P_2$ respectively. Moreover, in fact there are infinitely many supporting hyperplanes at $\pm P_1$ and $\pm P_2$. To prove this we will focus on the point $P_1=(1,0,0)$ since the other cases are similar. We just need to consider the family of hyperplanes $H_\theta$ passing through $(1,0,0)$ and parallel to the vectors $(0,1,0)$ and $(\sin\theta,0,\cos\theta-1)$ with $\theta\in (0,\pi/2)$ (see Figure~\ref{fig:hyperplane1}), that is, 
$$
H_\theta:\cos\theta(a-1)+(1-\sin\theta)c=\left|
\begin{array}{ccc}
a-1&b&c\\
\sin\theta-1&0&\cos\theta\\
0&1&0
\end{array}
\right|=0.
$$
Obviously ${\mathsf B}_{m,n}^h\cap H_\theta=\{(1,0,0)\}$ for all $\theta\in(0,\pi/2)$.

\begin{figure}
	\centering
	\includegraphics[height=.6\textwidth,keepaspectratio=true]{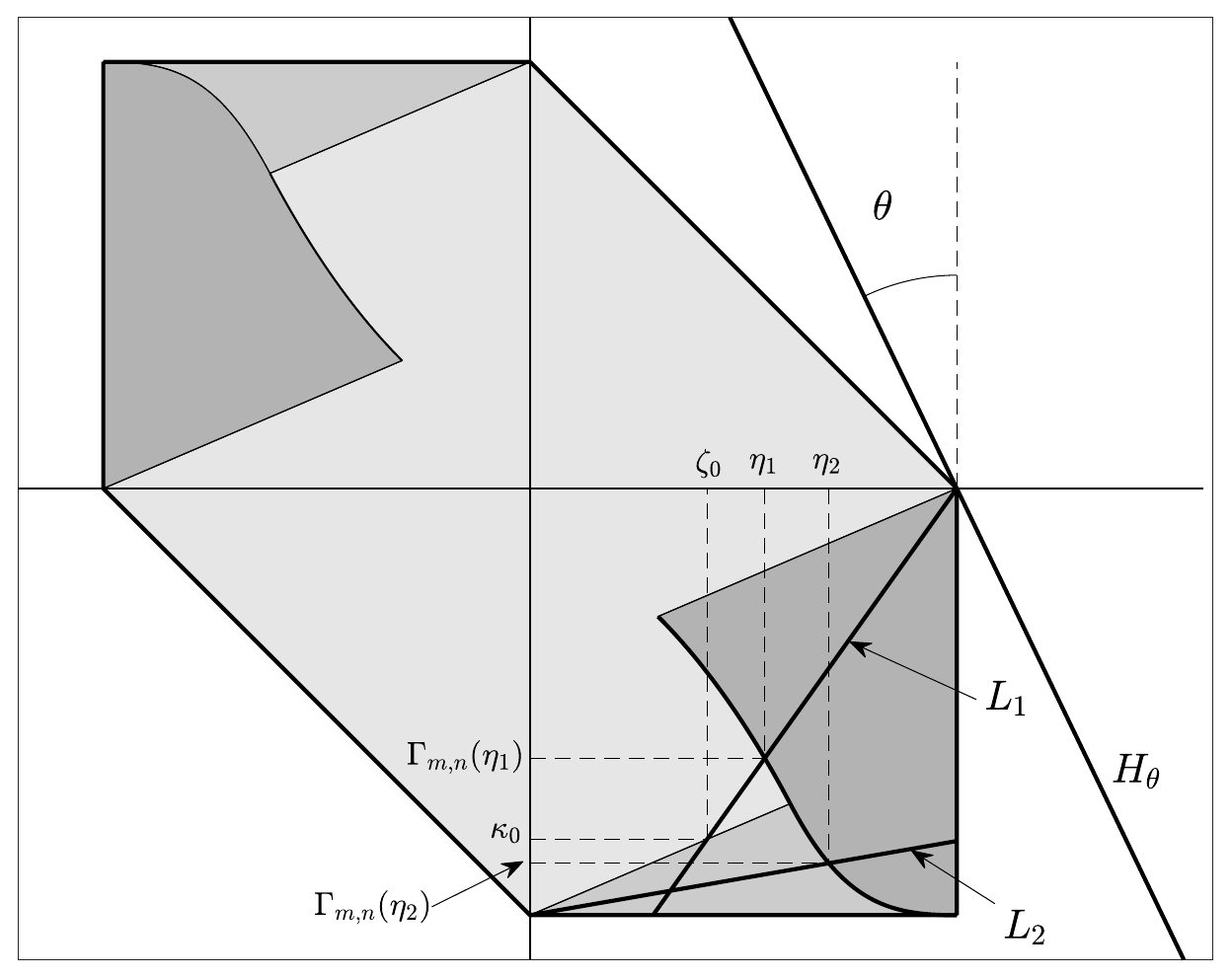}
	\caption{Representation of the projection onto the $ac$ plane of a supporting hyperplane $H_\theta$ to ${\mathsf B}_{m,n}^h$ at $(1,0,0)$. Here we have considered the case $m=10$, $n=3$ and $\theta=\pi/7$. Also, we have represented the segments $L_1$ and $L_2$ corresponding to two choices of $\eta$, namely $\eta_1=0.55$ and $\eta_2=0.7$, in the proof of Theorem~\ref{thm_extreme_points}. The point $(\zeta_0,\kappa_0)$ is the intersection of $L_1$ and the line $c=\lambda_0a-1$.}\label{fig:hyperplane1}
\end{figure}

As for the 8 curves in \eqref{eq:Gamma_ext_points} and \eqref{eq:Upsilon_ext_points}, it is enough to pay attention only to two cases due to the numerous existing symmetries. For instance take
\begin{itemize}
	\item[(a)] $\left(a, F_{m,n}(a,\Gamma_{m,n}(a)),\Gamma_{m,n}(a)\right)$ with $a\in[a_0,a_1]$.
	
	\item[(b)] $\left(a, F_{m,n}(a,\Upsilon_{m,n}(a)),\Upsilon_{m,n}(a)\right)$ with $a\in[a_1,1]$.
\end{itemize}

Let us study the curve appearing in (a). For simplicity we put
\begin{align*}
f_1(a)&=F_{m,n}(a,\Gamma_{m,n}(a)),\\
g_1(a)&=\Gamma_{m,n}(a),
\end{align*}
ending up with the curve
$$
C_1(a)=(a,f_1(a),g_1(a))
$$
for $a\in[a_0,a_1]$. Taking an arbitrary $\eta_1\in [a_0,a_1]$, a supporting plane to ${\mathsf B}_{m,n}^h$ at $C_1(\eta_1)$ must be tangent to the curve $C_1$ at $C_1(\eta_1)$, or equivalently, parallel to $C_1'(\eta)=(1,f_1'(\eta_1),g_1'(\eta_1))$. We need another vector to construct a plane passing through $C_1(\eta_1)$. Let $\gamma_1(a)=\lambda (a-1)$ with 
$$
\lambda=\frac{\Gamma_{m,n}(\eta_1)}{\eta_1-1}.
$$
The slope $\lambda$ has been chosen so that the line in the $ac$-plane given by $c=\gamma_1(a)$ with $a\in{\mathbb R}$ passes through $(\eta_1,\Gamma_{m,n}(\eta_1))$ and $(1,0)$. Consider the segment $L_1$ which results in intersecting the line $\gamma_1$ and the projection of ${\mathsf B}_{m,n}^h$ onto the $ac$ plane (see Figure~\ref{fig:hyperplane1}). Also, consider the plane $\Pi_1$ that contains $L_1$ and is parallel to the $b$ axis. Obviously $\Pi_1$ contains $(1,0,0)$. Recall we have shown above in this proof that $F_{m,n}$ is linear on the segment joining $(1,0)$ and $(\eta_1,\Gamma_{m,n}(\eta_1))$. Therefore the segment $L^*_1$ joining the points $(1,0,0)$ and $C_1(\eta_1)$ is contained in the graph of $F_{m,n}$ restricted to $L_1$. Since $F_{m,n}$ is affine on $W_{m,n}$, the restriction of $F_{m,n}$ to $L_1\cap W_{m,n}$ is a segment, say $L^\#_1$. Therefore the graph of $F_{m,n}$ restricted to $L_1$ has two confluent segments at $(\eta_1,\Gamma_{m,n}\eta_1))$, namely $L^*_1$ and $L^\#_1$. We have represented this situation in Figure~\ref{fig:hyperplanes3} where we have depicted the plane $\Pi_1$ with the graph of the restriction of $F_{m,n}$ to $L_1$. The remaining vector needed to construct a supporting hyperplane to ${\mathsf B}_{m,n}^h$ at $C_1(\eta_1)$ can be chosen in many ways. We call that vector $u_\delta$ and the resulting supporting hyperplane $H_\delta$ is:
$$
H_\delta: 
\left|
\begin{array}{c}
(a,b,c)-C_1(\eta_1)\\
C'_1(\eta_1)\\
u_\delta
\end{array}
\right|=0.
$$ 
The nature of $u_\delta$ and the index $\delta$ are revealed next. First consider two unitary vectors $v_{\alpha_1}$ and $v_{\beta_1}$ parallel to $L_1^\#$ and $L_1^*$ respectively. Here $\alpha_1$ and $\beta_1$ are, respectively, the angles between $L_1^\#$ and the plane $b=0$ and between $L_1^*$ and the plane $b=0$ (see Figure~\ref{fig:hyperplanes3}).
Hence, the desired vector $u_\delta$ can be expressed in canonical coordinates of the plane $\Pi_1$ as
$$
u_\delta=(\cos(\delta),\sin(\delta))
$$
with $\delta\in (\beta_1,\alpha_1)$. Observe that we have not included the cases $\delta=\alpha_1,\beta_1$ because ${\mathsf B}_{m,n}^h\cap H_{\alpha_1}=L^\#_1$ whereas ${\mathsf B}_{m,n}^h\cap H_{\beta_1}=L^*_1$.

\begin{figure}
	\centering
	\includegraphics[height=.75\textwidth,keepaspectratio=true]{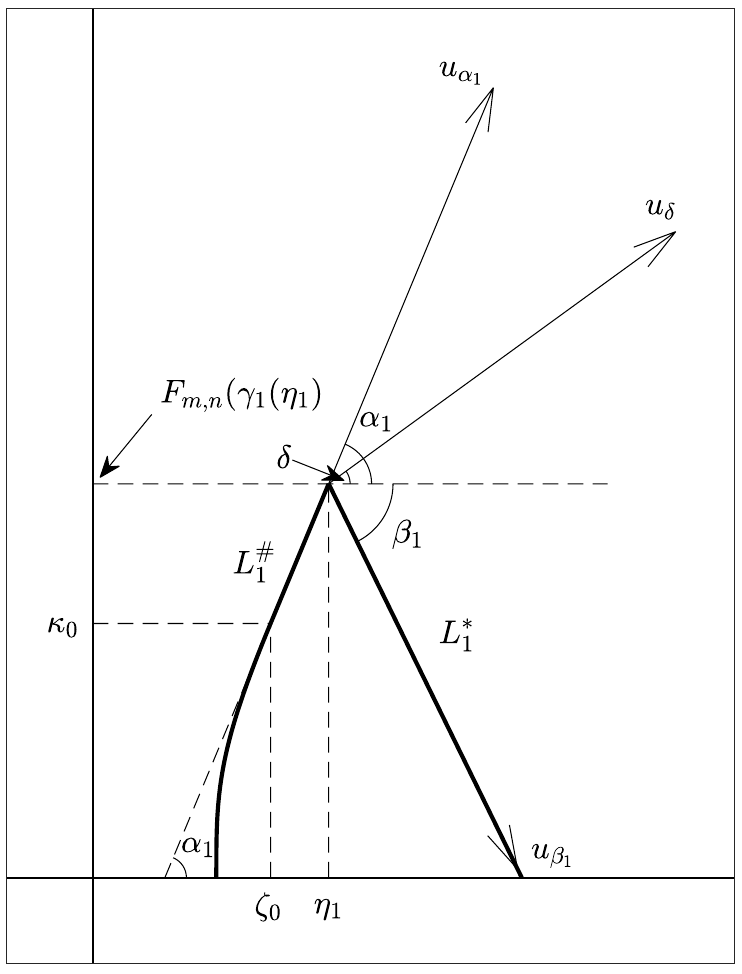}
	\caption{Representation of the plane $\Pi_1$ with the restriction of $F_{m,n}$ to $L_1$. The mapping $[\zeta_0,\eta_1]\ni a\mapsto F_{m,n}(\gamma_1(a))$ is linear and its graph is the segment $L_1^\#$. Also $[\eta_1,1]\ni a\mapsto F_{m,n}(\gamma_1(a))$ is linear too and its graph is the segment $L_1^*$. Here we have considered the case $m=10$, $n=3$, $\eta_1=0.55$ and $\delta=\pi/5$. }\label{fig:hyperplanes3}
\end{figure}

Finally we consider the curve in (b). Putting 
\begin{align*}
f_2(a)&=F_{m,n}(a,\Upsilon_{m,n}(a)),\\
g_2(a)&=\Upsilon_{m,n}(a),
\end{align*}
we end up with the curve
$$
C_2(a)=(a,f_2(a),g_2(a))
$$
for $a\in[a_1,1]$. Taking an arbitrary $\eta_2\in [a_1,1]$, a supporting plane to ${\mathsf B}_{m,n}^h$ at $C_2(\eta_2)$ must be tangent to the curve $C_2$ at $C_2(\eta_2)$, or equivalently, parallel to $C_2'(\eta_2)=(1,f_2'(\eta_2),g_2'(\eta_2))$. We need another vector to construct a plane passing through $C_2(\eta_2)$. Let $\gamma_2(a)=\lambda a-1$ with 
$$
\lambda=\frac{\Upsilon_{m,n}(\eta_2)+1}{\eta_2}.
$$
Notice that the slope $\lambda$ has been chosen so that the segment $L_2$ in the $ac$-plane given by $c=\gamma_2(a)$ with $a\in[0,1]$ passes through $(\eta_2,\Upsilon(\eta_2))$, starting at $(0,-1)$. Consider the plane $\Pi_2$ that contains $L_2$ and is parallel to the $b$ axis. Obviously $\Pi_2$ contains $(0,0,-1)$. Recall we have proved above in this proof that $F_{m,n}$ is linear on the segment joining $(0,-1)$ and $(\eta_2,\Upsilon(\eta_2))$, as can be seen in Figure~\ref{fig:hyperplanes2} where we have depicted the plane $\Pi_2$ with the representation of the restriction of $F_{m,n}$ to $L_2$. Let us call $L_2^*$ the segment joining $(-1,0,0)$ and $(\eta_2,F_{m,n}(\eta_2,\Upsilon_{m,n}(\eta_2)),\Upsilon_{m,n}(\eta_2))$. The remaining vector needed to construct a supporting hyperplane to ${\mathsf B}_{m,n}^h$ at $C_2(\eta_2)$ can be chosen in many ways. If we call that vector $v_\omega$, then the desired plane would be given by
$$
H_\omega: 
\left|
\begin{array}{c}
(a,b,c)-C_2(\eta_2)\\
C'_2(\eta_2)\\
v_\omega
\end{array}
\right|=0.
$$ 
The meaning of $v_\omega$ and the index $\omega$ will be explained in the rest of the proof. First consider the angle $\alpha_2$ between the segment $L^*_2$ and the plane $b=0$, that is,
$$
\alpha_2=\arctan(F_{m,n}(C_2(\eta_2))/\eta_2).
$$
Also let us choose a vector $v_{\beta_2}$ parallel to the tangent line to the graph of $[0,1]\ni a\mapsto F_{m,n}(\gamma_2(a))$ at $a=\eta_2$, whose slope $s_2$ in the plane $\Pi_2$ is given by 
$$
s_2=\frac{d}{da}\left[F_{m,n}(\gamma_2(a))\right]|_{a=\eta_2}=-J_{m,n}\left(\frac{1-\lambda \eta_2}{1-\eta_2}\right)^\frac{n}{m}\left(1-\frac{n}{m}\frac{1-\lambda}{1-\lambda \eta_2}\right).
$$
Hence, in canonical coordinates of the plane $\Pi_2$, $v_{\beta_2}$ could be expressed as 
$$
v_{\beta_2}=(\cos(\beta_2),\sin(\beta_2))
$$
where
$$
\beta_2=\arctan(s_2).
$$
Finally, $v_\omega$ could be chosen to be any unitary vector between $v_{\alpha_2}$ and $v_{\beta_2}$ in the plane $\Pi_2$, that is
$$
v_\omega=(\cos(\omega),\sin(\omega))
$$
with $\omega\in[\beta_2,\alpha_2)$ in coordinates of $\Pi_2$. Observe that we have not included the case $\omega=\alpha_2$ because $H_{\alpha_2}\cap {\mathsf B}_{m,n}^h=L_2^*$. We recommend to keep an eye on Figure~\ref{fig:hyperplanes2} to follow the last part of the proof.

\begin{figure}
	\centering
	\includegraphics[height=.75\textwidth,keepaspectratio=true]{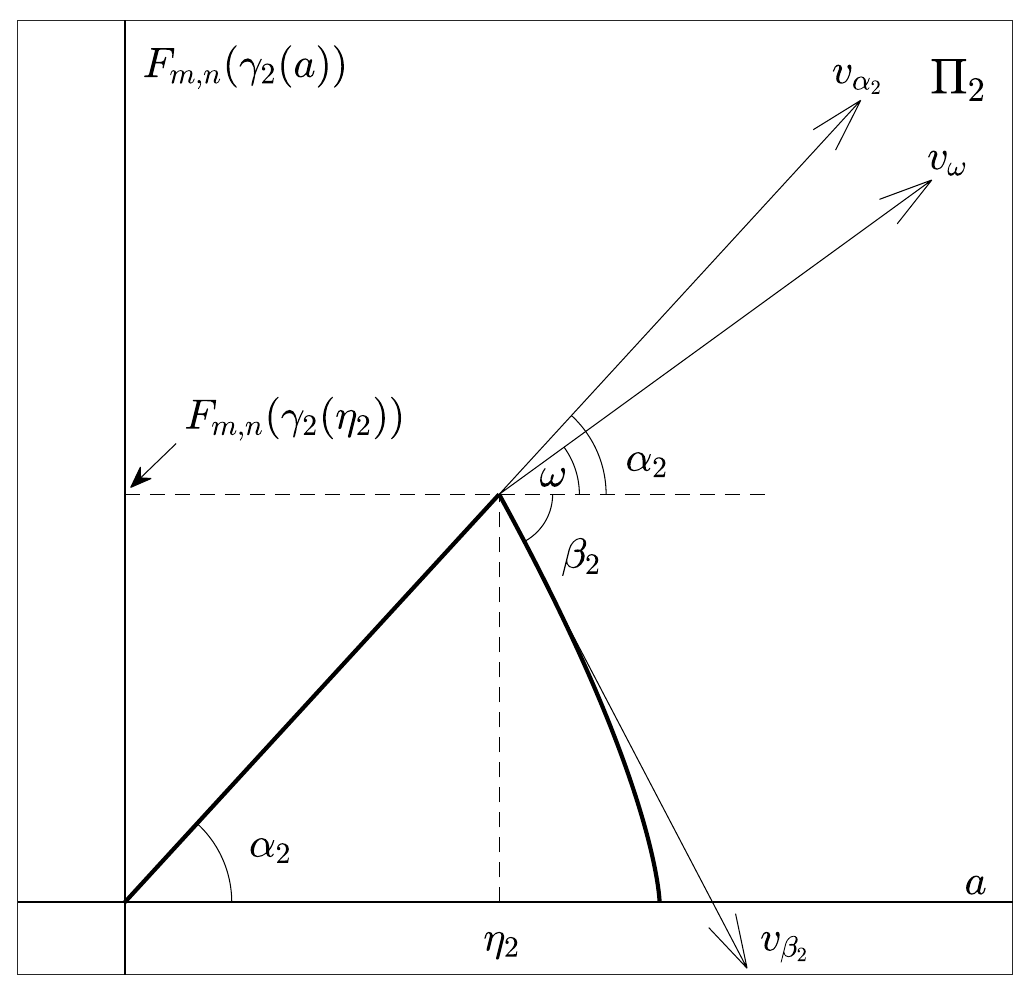}
	\caption{Representation on the plane $\Pi_2$ of the mapping $[0,1]\ni a\mapsto F_{m,n}(\gamma_2(a))$. Here we have considered the case $m=10$, $n=3$, $\eta_2=0.7$ and $\omega=\pi/5$. }\label{fig:hyperplanes2}
\end{figure}
\end{proof}

\end{document}